\documentclass[reqno,11pt]{amsart}

\usepackage{amssymb}
\usepackage[margin=1.25in,letterpaper]{geometry}
\usepackage{layout}
\usepackage{tikz-cd}
\usepackage{amsmath,amsthm,mathrsfs} 
\usepackage{marvosym}
\usepackage{xcolor}
\usepackage{colonequals}
\usepackage{enumitem}
\usepackage{mathdots}
\usepackage{framed}
\usepackage{url}
\usepackage{multicol}
\usepackage{fancyhdr}
\usepackage{nccmath}
\usepackage{color}
\usepackage{xparse}
\usepackage{bbm}
\usepackage[unicode,colorlinks]{hyperref}
\usepackage{microtype}
\usepackage{marginnote}
\usepackage{mathtools}

     \definecolor{dark-red}{rgb}{0.54,0,0}
     \definecolor{dark-green}{rgb}{0,0.54,0}
     \definecolor{dark-magenta}{rgb}{0.54,0,0.54}
     \definecolor{dark-cyan}{rgb}{0,0.54,0.54}
     \hypersetup{%
         unicode=true,
         breaklinks=true,
         pdfcreator=,
         pdfproducer=,
         pdfstartview =, 
         pdfpagelayout = ,
         pdfhighlight = /O,
         linkcolor=dark-red, 
         linkbordercolor=dark-red, 
         anchorcolor=red, 
         citecolor=dark-green, 
         citebordercolor=dark-green, 
         filecolor=dark-cyan, 
         filebordercolor=dark-cyan, 
         menucolor=dark-red, 
         menubordercolor=dark-red, 
         runcolor=dark-cyan, 
         runbordercolor=dark-cyan, 
         urlcolor=dark-cyan, 
         urlbordercolor=dark-cyan, 
}

\newcommand{\Affine}{\mathbb{A}}

\newcommand\FF{\protect\mathbb{F}}

\newcommand\NN{\protect\mathbb{N}}
\newcommand\QQ{\protect\mathbb{Q}}

\newcommand\ZZ{\protect\mathbb{Z}}

\newcommand\GG{\protect\mathbb{G}}

\newcommand\bN{\mathbb{N}}

\newcommand\bW{\mathbb{W}}

\newcommand\bZ{\mathbb{Z}}

\newcommand\sF{\mathscr{F}}

\newcommand\sI{\mathscr{I}}

\newcommand\sX{\mathscr{X}}

\newcommand\cA{\mathcal{A}}

\newcommand\cG{\mathcal{G}}
\newcommand\cI{\mathcal{I}}

\newcommand\cO{\mathcal{O}}

\newcommand\cR{\mathcal{R}}

\newcommand\cW{\mathcal{W}}

\DeclareMathOperator{\Mat}{Mat}
\DeclareMathOperator{\SL}{SL}

\DeclareMathOperator\GL{GL}
\DeclareMathOperator\Hom{Hom}
\DeclareMathOperator\Tr{Tr}

\DeclareMathOperator\Gal{Gal}

\DeclareMathOperator\Char{char}

\DeclareMathOperator\sgn{sgn}

\DeclareMathOperator\tr{tr}

\DeclareMathOperator\Ind{Ind}

\DeclareMathOperator{\diag}{diag}

\DeclareMathOperator\Frac{Frac}

\DeclareMathOperator{\Norm}{N}
\DeclareMathOperator{\rec}{rec}

\newcommand{\from}{\colon}

\DeclarePairedDelimiter{\floor}{\lfloor}{\rfloor}

\theoremstyle{theorem} \newtheorem{proposition}{Proposition}[section]
\theoremstyle{definition} \newtheorem{definition}[proposition]{Definition}
\theoremstyle{theorem} \newtheorem{lemma}[proposition]{Lemma}
\theoremstyle{remark} \newtheorem{remark}[proposition]{Remark}
\theoremstyle{remark} 
\theoremstyle{remark} 
\theoremstyle{definition} 
\theoremstyle{definition} \newtheorem{notation}[proposition]{Notation}
\theoremstyle{theorem} \newtheorem*{displaytheorem}{Theorem}
\theoremstyle{theorem} 
\theoremstyle{theorem} \newtheorem{theorem}[proposition]{Theorem}
\theoremstyle{theorem} \newtheorem{corollary}[proposition]{Corollary}
\theoremstyle{definition} 
\theoremstyle{theorem} 
\theoremstyle{remark} 
\theoremstyle{definition} 
\theoremstyle{definition} 
\theoremstyle{definition} 
\theoremstyle{definition} 
\theoremstyle{remark} \newtheorem*{claim*}{Claim}
\theoremstyle{remark} 
\theoremstyle{theorem} 
\theoremstyle{theorem} 
\theoremstyle{definition} 
\theoremstyle{definition} 
\theoremstyle{theorem} 
\theoremstyle{remark} 
\theoremstyle{definition} 
\theoremstyle{remark}
\theoremstyle{theorem} \newtheorem*{main}{Main Theorem}

\newcommand\Weil{\mathscr{W}}

\newcommand\ur{\text{nr}}

\newcommand\OH{\mathcal{O}}

\newcommand\Khat{\widehat{K}^{\ur}}
\newcommand\TT{\mathbb{T}}
\newcommand\UU{\mathbb{U}}
\newcommand\BB{\mathbb{B}}

\DeclareMathOperator\Fr{Fr}

\DeclareMathOperator\pr{pr}

\newcommand\Loc{\mathcal{L}}

\newcommand\F{\FF_{q^n}}
\newcommand\Unip{U_h^{n,q}}
\newcommand\UnipF{\Unip(\F)}

\DeclareMathOperator\Nrd{Nrd}

\newcommand\Unipk{U_{h,k}^{n,q}}
\newcommand\UnipkF{\Unipk(\F)}

 {\endMakeFramed\noindent\ignorespacesafterend}
 
 \numberwithin{equation}{section}
 
\title[Deligne--Lusztig Constructions for Division Algebras]{Deligne--Lusztig Constructions for Division Algebras and the Local Langlands Correspondence, II}
\author{Charlotte Chan}
\thanks{This work was partially supported by NSF grants DMS-0943832 and DMS-1160720.}
\email{charchan@umich.edu}

\begin{document}

\maketitle

\begin{abstract}
In 1979, Lusztig proposed a cohomological construction of supercuspidal representations of reductive $p$-adic groups, analogous to Deligne--Lusztig theory for finite reductive groups. In this paper we establish a new instance of Lusztig's program. Precisely, let $X$ be the $p$-adic Deligne--Lusztig ind-scheme associated to a division algebra $D$ of invariant $k/n$ over a non-Archimedean local field $K$. 
We study the $D^\times$-representations $H_\bullet(X)$ by establishing a Deligne--Lusztig theory for families of finite unipotent groups that arise as subquotients of $D^\times$. There is a natural correspondence between quasi-characters of the (multiplicative group of the) unramified degree-$n$ extension of $K$ and representations of $D^{\times}$ given by $\theta \mapsto H_\bullet(X)[\theta]$. For a broad class of characters $\theta,$ we show that the representation $H_\bullet(X)[\theta]$ is irreducible and concentrated in a single degree. Moreover, we show that this correspondence matches the bijection given by local Langlands and Jacquet--Langlands. As a corollary, we obtain a geometric realization of Jacquet--Langlands transfers between representations of division algebras.
\end{abstract}

\tableofcontents

\section{Introduction}\label{s:introduction}

Deligne--Lusztig theory \cite{DL76} gives a geometric description of the irreducible representations of finite groups of Lie type. In \cite{L79}, Lusztig suggests an analogue of Deligne--Lusztig theory for $p$-adic groups $G$. For an unramified maximal torus $T \subset G$, he introduces a certain infinite-dimensional variety which has a natural action of $T \times G$. Though it is not known in general, when $G$ is a division algebra, one can define $\ell$-adic homology groups $H_i(X)$ functorial for this action. One therefore obtains a correspondence $\theta \mapsto H_i(X)[\theta]$ between characters of $T$ and representations of $G$. In this paper, we study this correspondence and give a description from the perspective of the local Langlands and Jacquet--Langlands correspondences.

Let $K$ be a non-Archimedean local field with ring of integers $\cO_K$ and residue field $\FF_q = \cO_K/\pi$ for a fixed uniformizer $\pi$, and let $L \supset K$ be the unramified extension of degree $n$ with ring of integers $\cO_L$. A smooth character $\theta \from L^\times \to \QQ_\ell^\times$ is said to be \textit{primitive of level $h$} if $h$ is the smallest integer such that $\theta$ and $\theta/\theta^\gamma$ for $1 \neq \gamma \in \Gal(L/K)$ are trivial on $1 + \pi^h \cO_L$. This is equivalent to saying $(L,\theta)$ is a minimal admissible pair. To a primitive character $\theta \from L^\times \to \overline \QQ_\ell^\times$, one can associate a smooth irreducible $n$-dimensional representation $\sigma_{\xi\theta}$ of the Weil group $\cW_K$ of $K$, which corresponds via local Langlands to an irreducible supercuspidal representation $\pi_\theta$ of $\GL_n(K)$, which finally corresponds via Jacquet--Langlands to an irreducible representation $\rho_\theta$ of $D^\times,$ where $D = D_{k/n}$ is the central division algebra of invariant $k/n$ over $K$. For any $m \in \ZZ$, let $m^+ = \max\{m,0\}$.

\begin{main}
Let $\theta \from L^\times \to \overline \QQ_\ell^\times$ be a primitive character of level $h$. Then
\begin{equation*}
H_i(X)[\theta] = \begin{cases}
\rho_\theta & \text{if $i = r_\theta \colonequals (n-1)(h-k)^+$,} \\
0 & \text{otherwise.}
\end{cases}
\end{equation*}
\end{main}

Pictorially, 
\begin{equation*}
\begin{tikzcd}[column sep=tiny,row sep=small]
{\theta} \arrow[mapsto]{ddd}[left]{\text{$p$-adic Deligne--Lusztig}} & & {\theta} \arrow[mapsto]{d} & & {\mathfrak X} \arrow{d} \\
{} & & {\sigma_{\xi\theta}} \arrow[mapsto]{d} & & {\cG_K(n)} \arrow{d}{\text{ Local Langlands}} \\
{} & & {\pi_\theta} \arrow[mapsto]{d} & & {\cA_K(n)} \arrow{d}{\text{ Jacquet--Langlands}} \\ 
H_{r_\theta}(X)[\theta] & \cong & {\rho_\theta} & & {\cA'{}_K(n)}
\end{tikzcd}
\end{equation*}
where
\begin{align*}
\mathfrak X &\colonequals \{\text{primitive characters $L^\times \to \overline \QQ_\ell^\times$}\} \\
\cG_K(n) &\colonequals \{\text{smooth irreducible dimension-$n$ representations of the Weil group $\cW_K$}\} \\
\cA_K(n) &\colonequals \{\text{supercuspidal irreducible representations of $\GL_n(K)$}\} \\
\cA'{}_K(n) &\colonequals \{\text{smooth irreducible representations of $D^\times$}\}
\end{align*}

\subsection{What is known.}

Lusztig's definition in \cite{L79} has a natural analogue for groups over $\cO_K$ and its quotients $\cO_K/\pi^h$. This is described in \cite{L04}, where Lusztig also explicitly describes the resulting representations for $\SL_2(\cO_K/\pi^h)$ when $h \leq 2$. We note that our paper is in the setting of division algebras, whose finite reductive quotient is trivial, so the work of \cite{L04} does not play a role.

We now give a survey of known results on $p$-adic Deligne--Lusztig varieties for division algebras. In the next two sections, we let $G = D_{k/n}^\times$ and $T = L^\times$. Additionally, $G^1$ and $T^1$ denote the norm-1 elements of $G$ and $T$, and let $X$ and $X^1$ be the Deligne--Lusztig construction associated to $G$ and $G^1$. Write $H_i(X) = H_i(X, \overline \QQ_\ell)$ and $H_c^i(X) = H_c^i(X, \overline \QQ_\ell)$.

In \cite{L79}, Lusztig proves that when $k=1$, the virtual $G^1$-representations $\sum (-1)^i H_i(X^1)[\theta]$ are (up to sign) irreducible and mutually nonisomorphic.

In analogy with the behavior of classical Deligne--Lusztig varieties, one expects the homology groups $H_i(X)[\theta]$ to vanish outside a single degree. Additionally, one hopes to get a description of the irreducible representations arising from these homology groups. 

There exists a unipotent group scheme $\Unipk$ over $\FF_q$ such that $\UnipkF$ is isomorphic to a subquotient of $D_{k/n}^\times$. The study of $H_i(X)[\theta]$ reduces to the study of certain subschemes $X_h \subset \Unipk$ endowed with a left action by $(1 + \pi \cO_L)/(1 + \pi^h \cO_L)$ and a right action by $\UnipkF$. When $k=1$, these definitions were established in \cite{B12} for arbitrary $K$ if $h \leq 2$ and for $K$ of equal characteristic if $h > 2$. These definitions can be extended to the mixed characteristic case for arbitrary level $h$ and invariant $k/n$, and we do so in this paper.

In \cite{BW14}, Boyarchenko and Weinstein study the representations $H_c^i(X_2)$ when $k = 1$ (see Theorem 4.7 of \textit{op.\ cit.}). This comprises  one of the main ingredients in studying the cohomology of the Lubin--Tate tower. In \cite{BW13}, they specialize this result to the primitive case to give an explicit and partially geometric description of local Langlands correspondences. In \cite{B12}, Boyarchenko uses the representations $H_c^i(X_2)$ to prove that for any smooth character $\theta \from T \to \QQ_\ell^\times$ of level $\leq 2$, the representation $H_i(X)[\theta]$ vanishes outside a single degree and gives a description of this representation (see Theorem 5.3 of \textit{op.\ cit.}). Moreover, he shows that if $\theta$ is primitive, $H_i(X)[\theta]$ is irreducible in the nonvanishing degree.

In contrast to the structure of the Lubin--Tate tower, we need to understand the cohomology of $X_h$ for all $h$ to understand high-depth representations arising in Deligne--Lusztig constructions. Outside the equal characteristics case for $k = 1$, $n = 3$, and $h = 3$ (see Theorem 5.20 of \cite{B12}),  this was completely open.

In \cite{C14}, we study $X_h$ in the equal characteristics case for arbitrary $h$, assuming $n = 2$ and $\chi$ is primitive. We prove irreducibility of $H_c^i(X_2)[\chi]$ and vanishing outside a single degree. In addition we prove a character formula in the form of a branching rule for representations of the finite unipotent group $U_{h,1}^{2,q}(\FF_{q^2})$, a subquotient of the quaternion algebra. Using this, we are able to study the representations $H_i(X)[\theta]$ for primitive characters $\theta$.

In this paper, we generalize this work to arbitrary $n$, arbitrary $k$, and arbitrary $K$, thereby removing all assumptions outside primitivity. We take a more conceptual approach that allows us to bypass many of the computations needed in \cite{C14}. As a corollary, we obtain a geometric realization of Jacquet--Langlands transfers between representations of division algebras.

\begin{remark}
In the special case that $n = 2$ and $\Char K > 0$, the $p$-adic Deligne--Lusztig constructions we study in this paper and its prequel \cite{C14} are cut out by equations that look similar to the equations defining certain covers of affine Deligne--Lusztig varieties. This was observed by Ivanov in Section 3.6 of \cite{I15}. \hfill $\Diamond$
\end{remark}

\subsection{Outline of this paper}

In Section \ref{s:definitions}, we introduce the unipotent groups $\Unipk$ together with a certain subgroup $H \subset \Unipk$. The finite groups $\UnipkF$ and $H(\F)$ are subquotients of $G$ and $T$, respectively. We then define a certain subvariety $X_h \subset \Unipk$, whose relation to the $p$-adic Deligne--Lusztig construction $X$ is as follows: $X$ can be identified with a set $\widetilde X$ endowed with an ind-scheme structure
\begin{equation*}
\widetilde X = \bigsqcup_{m \in \ZZ} \varprojlim_h \widetilde X_h^{(m)},
\end{equation*}
where $\widetilde X_h^{(0)}$ is the disjoint union of $q^n-1$ copies of $X_h(\overline \FF_q)$. Roughly speaking, the action of $T \times G$ on $\widetilde X$ has essentially two behaviors: there is an action on each $\widetilde X_h^{(m)}$, and there is an action permuting these pieces. In order to understand the $(T \times G)$-representations arising from $H_i(X)$, one must understand these two actions. The former is captured by the action of $H(\F) \times \UnipkF$ on $X_h$; the latter was studied by Boyarchenko in \cite{B12} (see Proposition 5.19 of \textit{op.\ cit.} for the equal characteristics, $k=1$ case). 

Let $\cA$ denote the set of primitive characters of $H(\F)$. Let $\cG$ denote the set of irreducible representations of $\UnipkF$ whose central character has trivial $\Gal(L/K)$-stabilizer. In Section \ref{s:reps}, we give a correspondence $\chi \mapsto \rho_\chi$ from $\cA$ to $\cG$. When $k = 1$, this construction matches that of Corwin \cite{C74}.

In Section \ref{s:jugglingdesc} we study the geometry of $X_h$ using a combinatorial notion known as \textit{juggling sequences.} We prove in Theorem \ref{l:Xhpolys} that the varieties $X_h$ are affine varieties defined by the vanishing of polynomials whose monomials are indexed by juggling sequences. By studying the combinatorics of these objects, we are able to prove structural lemmas crucial to the analysis of $H_c^i(X_h)$.

Section \ref{s:cohomdesc} is concerned with combining the general algebro-geometric results of Section \ref{s:alggeom}, the representation-theoretic results of Section \ref{s:reps}, and the combinatorial results of Section \ref{s:jugglingdesc}. In Theorem \ref{t:cohomdesc} that the correspondence $\chi \mapsto \rho_\chi$ is bijective and that every representation $\rho \in \cG$ appears in $H_c^i(X_h)$ with multiplicity $1$. In addition, we prove a character formula for the representations $H_c^i(X_h)[\chi]$ using the Deligne--Lusztig fixed point formula of \cite{DL76}.

Section \ref{s:divalg} is devoted to understanding two connections. The first, explained in Section \ref{s:DL}, is to unravel the relationship between the results of Section \ref{s:cohomdesc} and the representations of division algebras arising from $p$-adic Deligne--Lusztig constructions $\widetilde X$. 
The second, explained in Section \ref{s:LLC}, is to describe $H_i(X)[\theta]$ from the perspective of the local Langlands and Jacquet--Langlands correspondences. We use Theorem \ref{t:cohomdesc}, the trace formula established in Proposition \ref{p:vregtrace}, and a criterion of Henniart described in \cite{BW13} (see Proposition 1.5(b) of \textit{op.\ cit.}).

\begin{displaytheorem}[\ref{t:divalg}, \ref{c:JL}]
Let $\theta \from L^\times \to \overline \QQ_\ell^\times$ be a primitive character of level $h$ and let $\rho_\theta$ be the $D^\times$-representation corresponding to $\theta$ under the local Langlands and Jacquet--Langlands correspondences. Then 
\begin{equation*}
H_i(X)[\theta] = \begin{cases}
\rho_\theta & \text{if $i = (n-1)(h-k)^+$,} \\
0 & \text{otherwise.}
\end{cases}
\end{equation*}
Moreover, if $D$ and $D'$ are division algebras of invariant $k/n$ and $k'/n$ with associated Deligne--Lusztig constructions $X$ and $X'$, then the Jacquet--Langlands transfer of $H_\bullet(X)[\theta]$ is isomorphic to $H_\bullet(X)[\theta]$.
\end{displaytheorem}

Using the techniques developed in this paper, we have evidence to support that for nonprimitive characters $\theta \from L^\times \to \overline \QQ_\ell^\times$ of level $h$ with restriction $\chi \from U_L^1 \to \overline \QQ_\ell^\times$, the cohomology groups $H_c^i(X_h)[\chi]$ are irreducible and concentrated in a single non-middle degree. This implies that the homology groups $H_i(X)[\theta]$ are also concentrated in a single degree, though it it not expected that these representations are irreducible in general. We plan to investigate this in a future paper.

\subsection*{Acknowledgements}

I am deeply grateful to Mitya Boyarchenko for introducing me to this area of research. I'd also like to Jake Levinson for interesting conversations regarding Proposition \ref{p:claim1gen}.

\section{Definitions}\label{s:definitions}

Let $K$ be a non-Archimedean local field with residue field $\FF_q$ and fixed uniformiser $\pi$. If $K$ has characteristic $p$ (the equal characteristics case), then all the definitions below were already established in \cite{B12} and \cite{BW14}. If $K$ has characteristic $0$, the definitions we establish below are new.


We first recall the ring scheme $\bW_K$ of $\cO_K$-Witt vectors. (For an exposition, see for example \cite{FF13}.) For each $r \geq 0$, define the Witt polynomial
\begin{equation*}
W_r = \sum_{i=0}^r \pi^i X_i^{q^{r-i}} \in \cO_K[X_0, \ldots, X_r].
\end{equation*}
There exist polynomials $S_r, M_r \in \cO_K[X_0, \ldots, X_r, Y_0, \ldots, Y_r]$ such that
\begin{align*}
W_r(S(X,Y)) &= W_r(X) + W_r(Y), \\
W_r(M(X,Y)) &= W_r(X) \times W_r(Y).
\end{align*}
For any $\cO_K$-algebra $R$, define
\begin{equation*}
\bW_K(A) = A^\bN
\end{equation*}
with the addition and multiplication given by $X +_{\bW_K} Y = S(X,Y)$ and $X \times_{\bW_K} Y = M(X,Y)$. From now on, we will view $\bW_K$ as a scheme over $\cO_K/(\pi) = \FF_q$.

\begin{definition}
Let $A$ be any $\FF_q$-algebra. If $K$ has characteristic $p$, let $\bW(A) = A[\![\pi]\!]$, and if $K$ has characteristic $0$, let $\bW(A) = \bW_K(A)$. For any $h \in \bN$, we define $\bW_h(A) = \bW(A)/(\pi^h)$. It is clear that the functor $A \mapsto \bW_h(A)$ is representable by the affine space $\Affine^h$.
\end{definition}

\begin{definition}\label{d:unip}
For any $\FF_q$-algebra $A$, define a ring $\cR_{h,k,n,q}(A)$ as follows:
\begin{enumerate}[label=\textbullet]
\item
As a group under addition, $\cR_{h,k,n,q}(A) \colonequals \bW(A)[\tau]/(\tau^n - \pi^k, \pi^h, \pi^{h-k} \tau, \ldots, \pi^{h-k} \tau^{n-1})$.

\item
The multiplication structure on $\cR_{h,k,n,q}(A)$ is given by the following commutation rule: $\tau \cdot a = a^q \cdot \tau$ for any $a \in A$.
\end{enumerate}
Elements of $\cR_{h,k,n,q}$ can be written as $A_0 + A_1 \tau + \cdots + A_{n-1} \tau^{n-1}$ where $A_0 \in \bW_h$ and $A_i \in \bW_{h-k}$ for $i = 1, \ldots, n-1$. Then
\begin{equation*}
\cR_{h,k,n,q}^\times = \{A_0 + A_1 \tau + \cdots + A_{n-1} \tau^{n-1} : A_0 \in \bW_h^\times\} \subset \cR_{h,k,n,q}
\end{equation*}
and we define
\begin{equation*}
\Unipk = \{A_0 + A_1 \tau + \cdots + A_{n-1} \tau^{n-1} : A_0 = (1, *, \cdots, *) \in \bW_h^\times\} \subset \cR_{h,k,n,q}^\times.
\end{equation*}
It is clear that the functor $A \mapsto \Unipk(A)$ is representable by the affine space $\Affine^{(h-1)+(n-1)(h-k)^+}$.
\end{definition}


Define
\begin{equation*}
H \colonequals \{A_0 \in 1 + \bW_{h-1} \subset \bW_h\} \subset \Unipk.
\end{equation*}
Note that although $\bW_h$ is a commutative group scheme, the group scheme $H$ is not commutative. 
This will be standard notation throughout this paper outside Section \ref{s:alggeom}.

\begin{remark}\label{r:isoms}
Since we have natural isomorphisms $\bW(\FF_q) \cong \cO_K$ and $\bW(\F) \cong \cO_L$, we also have natural isomorphisms
\begin{flalign*}
\phantom{stuffstuffstuff}U_L^1/U_L^h &\overset{\sim}{\to} H(\F), && A_0 \mapsto A_0 \in \bW_h(\F) \\
\F \overset{\sim}{\to} U_L^{h-1}/U_L^h &\overset{\sim}{\to} H_{n(h-1)}(\F), && a \mapsto (1,0,\ldots,0,a). && 
\end{flalign*}
Note also that
\begin{flalign*}
\phantom{stuff}\UnipkF \cong U_D^1/U_D^{(h)}, \text{ where $U_D^{(h)} \colonequals (1 + P_L^h)(1 + P_L^{h-k}\Pi) \cdots (1 + P_L^{h-k}\Pi^{n-1}).$} &&\Diamond
\end{flalign*}
\end{remark}

\subsection{The varieties $X_h$} \label{s:Xhdef}

\begin{definition}
For any $\FF_q$-algebra $A$, let $\Mat_{h,k}(A)$ denote the ring of all $n$-by-$n$ matrices $B = (b_{ij})_{i,j=1}^n$ with $b_{ii} \in \bW_h(A)$, $b_{ij} \in \bW_{h-k}(A)$ for $i < j$, and $b_{ij} \in \pi^k \bW_{h-k}(A)$ for $i > j$. The determinant can be viewed as a multiplicative map $\det \from \Mat_{h,k}(A) \to \bW_h(A)$.
\end{definition}

For any $\FF_q$-algebra $A$, consider the morphism
\begin{equation*}
\iota_h \from \cR_{h,k,n,q}(A) \to \Mat_{h,k}(A)
\end{equation*}
given by
\begin{equation*}
\iota_{h,k}\left(\textstyle \sum A_i \tau^i\right) \colonequals
\left(\begin{matrix}
A_0 & A_1 & A_2 & \cdots & A_{n-1} \\
\pi^k \varphi(A_{n-1}) & \varphi(A_0) & \varphi(A_1) & \cdots & \varphi(A_{n-2}) \\
\pi^k \varphi^2(A_{n-2}) & \pi^k \varphi^2(A_{n-1}) & \varphi^2(A_0) & \cdots & \varphi^2(A_{n-3}) \\
\vdots & \vdots & \ddots & \ddots & \vdots \\
\pi^k \varphi^{n-1}(A_1) & \pi^k \varphi^{n-1}(A_2) & \cdots & \pi^k \varphi^{n-1}(A_{n-1}) & \varphi^{n-1}(A_0)
\end{matrix}\right)
\end{equation*}
where $\varphi \from \bW \to \bW$ is the $q$th Frobenius endomorphism.

Recall from \cite{B12} that the $p$-adic Deligne--Lusztig construction $X$ described in \cite{L79} can be identified with a certain set $\widetilde X$ which can be realized as the $\overline \FF_q$-points of an ind-scheme
\begin{equation*}
\widetilde X = \bigsqcup_{m \in \ZZ} \varprojlim_h \widetilde X_h^{(m)}.
\end{equation*}
Here, $\widetilde X_h^{(k)} \colonequals \widetilde X_h^{(0)} \cdot \Pi$ and $\widetilde X_h^{(n)} \colonequals \widetilde X_h^{(0)} \cdot \pi$, where for any $\overline \FF_q$-algebra $A$,
\begin{equation*}
\widetilde X_h^{(0)}(A) = \{\iota_{h,k}(\textstyle \sum a_i \tau^i) \in \Mat_h(A) : \text{$\det(\iota_{h,k}(\textstyle \sum a_i \tau^i))$ is fixed by $\varphi$}\}.
\end{equation*}

\begin{definition}
For any $\FF_q$-algebra $A$, define
\begin{equation*}
X_h(A) \colonequals \Unipk(A) \cap \iota_{h,k}^{-1}(\widetilde X_h^{(0)}(A)).
\end{equation*}
\end{definition}

\begin{remark}
Notice that $\widetilde X_h^{(0)}$ is a disjoint union of $q^n - 1$ copies of $X_h$. \hfill $\Diamond$
\end{remark}

\begin{remark}
Note that $X$, $\widetilde X$, $\widetilde X_h^{(m)}$, and $X_h$ all depend on Hasse invariant $k/n$ of the division algebra $D$, but we suppress this from the notation. \hfill $\Diamond$
\end{remark}

\subsection{Group actions}

The map $\iota_{h,k}$ has the following property, which we will refer to as Property $\ddagger$. If $A$ is an $\F$-algebra, then $\iota_{h,k}(xy) = \iota_{h,k}(x)\iota_{h,k}(y)$ for all $x \in \Unipk(A)$ and all $y \in \UnipkF$. Moreover, for $y \in \UnipkF$, the determinant of $\iota_{h,k}(y)$ is fixed by $\varphi$. It therefore follows that $X_h$ is stable under right-multiplication by $\UnipkF$. We denote by $x \cdot g$ the action of $g \in \UnipkF$ on $x \in X_h$.

Pick a generator $\zeta$ of $\F^\times$. The conjugation action of $\zeta$ on $\Unipk(A)$ stabilizes $X_h(A)$. This extends the right $\UnipkF$ action on $X_h$ to an action of the semidirect product $\F^\times \ltimes \Unipk(\F) \cong \cR_{h,k,n,q}^\times(\F)$ on $X_h$.

We now describe a left action of $H(\F)$ on $X_h$. We can identify $H(\F)$ with the set $\iota_{h,k}(H(\F))$. Note that by Property $\ddagger$, the map $\iota_{h,k}$ actually preserves the group structure of $H(\F)$, and since $\iota_{h,k}$ is injective, then $H(\F) \cong \iota_{h,k}(H(\F))$ as groups. Explicitly, this isomorphism is given by
\begin{equation*}
A_0 \mapsto \diag(A_0, \varphi(A_0), \ldots, \varphi^{n-1}(A_0)).
\end{equation*}
The left-multiplication action of $\iota_{h,k}(H(\F))$ on the $\Mat_{h,k}(A)$ stabilizes $X_h(A)$. We denote by $g * x$ the action of $g \in H(\F) \cong U_L^1/U_L^h$ on $x \in X_h$.\footnote{Warning: This is not the same as the left-multiplication action of $H(\F) \subset H(A)$ on $\Unipk(A)$.}

\begin{remark}
Let $Z(\UnipkF)$ denote the center of $\UnipkF$. This is a subgroup of $H(\F)$. By direct computation, one sees that the left action of $Z(\UnipkF) \subset H(\F)$ and the right action of $Z(\UnipkF) \subset \UnipkF$ coincide. Note also that the actions of $H(\F)$ and $\cR_{h,k,n,q}^\times(\F)$ commute. \hfill $\Diamond$
\end{remark}

\section{General principles: some algebraic geometry}\label{s:alggeom}

In this section, we prove some general algebro-geometric results that will allow us to compute certain cohomology groups via an inductive argument. We generalize the techniques of \cite{B12} from $\GG_a$ to a group scheme $Z$, which we now define\footnote{Note also that $Z$ is the subgroup scheme $H$ of $\Unip$ defined in Section \ref{s:definitions}.}. For any $\FF_{q^n}$-algebra $A$,
\begin{equation*}
Z(A) \colonequals \{A_0 \in 1 + \pi \bW_{h-1}(A)\} \subset \Unipk(A).
\end{equation*}

Let $G$ be an algebraic group over $\F$ and suppose that $Y \subset G$ is a (locally closed) subvariety defined over $\F$ and put $X = L_{q^n}^{-1}(Y)$, where $L_{q^n} \from G \to G$ is the Lang map given by $g \mapsto \Fr_{q^n}(g)g^{-1}.$ Let $H \subset G$ be any connected subgroup defined over $\F$ and let $\eta \from H(\F) \to \overline \QQ_\ell^\times$ be a character. Write $V_\eta = \Ind_{H(\F)}^{G(\F)}(\eta)$.

Consider the right-multiplication action of $H(\F)$ on $G$ and form the quotient $Q \colonequals G/(H(\F))$. The Lang map $L_{q^n} \from G \to G$ is invariant under right multiplication by $H(\F)$ and thus it factors through a morphism $\alpha \from Q \to G$. On the other hand, the quotient map $G \to Q$ is a right $H(\F)$-torsor, so the character $\eta$ yields a $\overline \QQ_\ell$-local system $\mathcal E_\eta$ of rank $1$ on $Q$. The following lemma is proved in \cite{B12}.

\begin{lemma}[Boyarchenko \cite{B12}]\label{l:B2.1}
There is a natural $\Fr_q$-equivariant vector-space isomorphism
\begin{equation*}
\Hom_{G(\FF_q)}(V_\chi, H_c^i(X, \overline \QQ_\ell)) \cong H_c^i(\alpha^{-1}(Y), \mathcal{E}_\chi|_{\alpha^{-1}(Y)}) \qquad \forall i \geq 0.
\end{equation*}
\end{lemma}

As in \cite{B12}, we now make two further assumptions under which the right-hand side of the isomorphism in Lemma \ref{l:B2.1} can be described much more explicitly. This will allow us to certain cohomology groups via an inductive argument. These two assumptions are:
\begin{enumerate}[label=\arabic*.]
\item
The quotient morphism $G \to G/H$ admits a section $s \from G/H \to G$.

\item
There is an algebraic group morphism $f \from H \to Z$ defined over $\F$ such that $\eta = \chi \circ f$ for a character $\chi \from Z(\F) \to \overline \QQ_\ell^\times$.
\end{enumerate}

Let $\Loc_\chi$ be the local system on $Z$ defined by $\chi$ via the Lang map $L_{q^n} \from Z \to Z$. The following lemma is proved in \cite{B12}.

\begin{lemma}[Boyarchenko \cite{B12}]\label{l:B2.2}
There is an isomorphism $\gamma \from (G/H) \times H \overset{\simeq}{\longrightarrow} Q$ such that $\gamma^* \mathcal E_\eta \cong (f \circ \pr_2)^*\Loc_\chi$ and $\alpha \circ \gamma = \beta$, where $\pr_2 \from (G/H) \times H \to H$ is the second projection and $\beta \from (G/H) \times H \to G$ is given by $\beta(x,h) = s(\Fr_{q^n}(x)) \cdot h \cdot s(x)^{-1}$.
\end{lemma}

Combining Lemma \ref{l:B2.1} and \ref{l:B2.2} together with the assumption that $\eta = \chi \circ f$ for a character $\chi \from Z(\F) \to \overline \QQ_\ell^\times$ and an algebraic group morphism $f \from H \to Z$ defined over $\F$, we obtain the following proposition.

\begin{proposition}\label{p:B2.3}
Assume that we are given the following data:
\begin{enumerate}[label=\textbullet]
\item an algebraic group $G$ with a connected subgroup $H \subset G$ over $\F$;

\item
a section $s \from G/H \to G$ of the quotient morphism $G \to G/H$;

\item
an algebraic group homomorphism $f \from H \to Z$;

\item
a character $\chi \from Z(\F) \to \overline \QQ_\ell^\times$;

\item
a locally closed subvariety $Y \subset G$.
\end{enumerate}
Set $X = L_{q^n}^{-1}(Y)$. The preimage of $Y$ under the Lang map $L_{q^n}(g) = \Fr_{q^n}(g) g^{-1}$. Then for each $i \geq 0$, we have a $\Fr_{q^n}$-compatible vector space isomorphism
\begin{equation*}
\Hom_{G(\F)}(\Ind_{H(\F)}^{G(\F)}(\chi \circ f), H_c^i(X, \overline \QQ_\ell)) \cong H_c^i(\beta^{-1}(Y), P^*\Loc_\chi).
\end{equation*}
Here, $\Loc_\chi$ is the local system on $Z$ corresponding to $\chi$, the morphism $\beta \from (G/H) \times H \to G$ is given by $\beta(x,h) = s(F_q(x)) \cdot h \cdot s(x)^{-1}$, and the morphism $P \from \beta^{-1}(Y) \to Z$ is the composition $\beta^{-1}(Y) \hookrightarrow (G/H) \times H \overset{\pr_2}{\longrightarrow} H \overset{f}{\to} Z$.
\end{proposition}

Our goal now is to prove the following crucial proposition. This is the proposition that gives us an inductive technique for calculating certain cohomology groups.

\begin{proposition}\label{p:B2.10}
Let $q$ be a power of $p$, let $n \in \NN$, and let $\chi \from Z(\FF_{q^n}) \to \overline \QQ_\ell^\times$ be primitive. Let $S_2$ be a scheme of finite type over $\FF_{q^n}$, put $S = S_2 \times \Affine^1$ and suppose that a morphism $P \from S \to Z$ has the form
\begin{equation*}
P(x,y) = g(f(x)^{q^{j_1}}y^{q^{j_2}}) \cdot g(f(x)^{q^{j_3}}y^{q^{j_4}})^{-1} \cdot P_2(x)
\end{equation*}
where 
\begin{enumerate}[label=\textbullet]
\item
$j_1 - j_2 = j_3 - j_4$ and $j_2 - j_4$ is not divisible by $n$,

\item
$f \from S_2 \to \GG_a$, $P_2 \from S_2 \to Z$ are two morphisms, and
\item
$g \from \Affine^1 \to Z$ is the morphism $z \mapsto (0, \ldots, 0, z)$. 
\end{enumerate}
Let $S_3 \subset S_2$ be the subscheme defined by $f = 0$ and let $P_3 = P_2|_{S_3} \from S_3 \to Z$. Then for all $i \in \ZZ$, we have 
\begin{equation*}
H_c^i(S, P^*\Loc_\chi) \cong H_c^{i-2}(S_3, P_3^*\Loc_\chi)(-1)
\end{equation*}
as vector spaces equipped with an action of $\Fr_{q^n}$, where the Tate twist $(-1)$ means that the action of $\Fr_{q^n}$ on $H_c^{i-2}(S_3, P_3^*\Loc_\chi)$ is multiplied by $q^n$.
\end{proposition}

\begin{proof} 
Let $\pr \from S = S_2 \times \Affine^1 \to S_2$ be the first projection, let $\iota \from S_3 \to S_2$ be the inclusion map, and let $\eta \from S \to Z$ be the morphism $(x,y) \mapsto g(f(x)^{q^{j_1}}y^{q^{j_2}}) \cdot g(f(x)^{q^{j_3}}y^{q^{j_4}})^{-1}$. The sheaf $\Loc_\chi$ is not a multiplicative local system on $Z$. However, since the image of $\eta$ lies in the center of the group scheme $Z$, then 
\begin{equation*}
P^* \Loc_\chi \cong (\eta^* \Loc_\chi) \otimes \pr^*(P_2^* \Loc_\chi).
\end{equation*}
Thus, by the projection formula,
\begin{equation*}
R \pr_!(P^* \Loc_\chi) \cong P_2^* \Loc_\chi \otimes R \pr_!(\eta^* \Loc_\chi) \qquad \text{in $D_c^b(S_2, \overline \QQ_\ell).$}
\end{equation*}
I now claim that
\begin{equation*}
R \pr_!(\eta^* \Loc_\chi) \cong \iota_!(\overline \QQ_\ell)[2](1) \qquad \text{in $D_c^b(S_2, \overline \QQ_\ell),$}
\end{equation*}
where $\overline \QQ_\ell$ denotes the constant local system of rank $1$. It is clear that once we have established this, the desired conclusion follows. We therefore spend the rest of the proof proving this.

The restriction of $\eta$ to $\pr^{-1}(S_3) \subset S_2$ is constant, so the restriction of the pullback $\eta^* \Loc_\chi$ to $\pr^{-1}(S_3)$ is a constant local system of rank $1$. Thus 
\begin{equation*}
\iota^* R \pr_!(\eta^* \Loc_\chi) \cong \overline \QQ_\ell[2](1) \qquad \text{in $D_c^b(S_3, \overline \QQ_\ell).$}
\end{equation*}
To complete the proof, we need show that $R \pr_!(\eta^* \Loc_\chi)$ vanishes outside $S_3 \subset S_2$. First notice that $\eta = g \circ \eta_0$ where $\eta_0 \from S \to \GG_a$ is defined as $(x,y) \mapsto f(x)^{q^{j_1}}y^{q^{j_2}} - f(x)^{q^{j_3}}y^{q^{j_4}}$. Let $\psi$ be the restriction of $\chi$ to $g(\GG_a) \subset Z$. Then
\begin{equation*}
\eta^* \Loc_\chi \cong \eta_0^* \Loc_\psi.
\end{equation*}
Here, $\Loc_\psi$ denotes the multiplicative local system on $\GG_a$ induced by $\psi$ via the Lang isogeny. It therefore suffices to show that $R \pr_!(\eta_0^* \Loc_\psi)$ vanishes outside $S_3 \subset S_2$. Now pick $x \in S_2(\overline \FF_q) \smallsetminus S_3(\overline \FF_q)$. By the proper base change theorem,
\begin{equation*}
R^i \pr_!(\eta_0^* \Loc_\psi)_x \cong H_c^i(\GG_a, f_x^* \Loc_\psi),
\end{equation*}
where $f_x \from \GG_a \to \GG_a$ is given by $y \mapsto f(x)^{q^{j_1}}y^{q^{j_2}} - f(x)^{q^{j_3}}y^{q^{j_4}}.$

As in the proof of Proposition 2.10 of \cite{B12}, we can write $\Loc_\psi = \Loc_z$ for some $z \in \F$. Since $\psi$ has conductor $q^n$, then $z$ has trivial $\Gal(\F/\FF_q)$-stabilizer. By Corollary 6.5 of \cite{B12}, we have $f_x^* \Loc_\psi \cong \Loc_{f_x^*(z)}$, where
\begin{equation*}
f_x^*(z) = f(x)^{q^{j_1}/q^{j_2}} z^{1/q^{j_2}} - f(x)^{q^{j_3}/q^{j_4}} z^{1/q^{j_4}} = f(x)^{q^{j_1 - j_2}}(z^{q^{-j_2}} - z^{q^{-j_4}}).
\end{equation*}
But $z^{q^{-j_2}} - z^{q^{-j_4}} \neq 0$ since by assumption $z \neq 0$ and $j_2 - j_4$ is not divisible by $n$ by assumption. Thus $f_x^* \Loc_\psi$ is a nontrivial local system on $\GG_a$ and $H_c^i(\GG_a, f_x^* \Loc_\psi) = 0$ for all $i \geq 0$.
\end{proof}

\begin{proposition}\label{p:B2.10alpha}
Suppose that $P \from S \to Z$ has the form
\begin{equation*}
P(x,y) = g(f(x)^{q^{j_1}} y^{q^{j_2}} - f(x)^{q^{j_3}} y^{q^{j_4}} + \alpha(x,y)^{q^n} - \alpha(x,y)) \cdot P_2(x)
\end{equation*}
for some morphism $\alpha \from S_2 \times \Affine^1 \to \GG_a$ defined over $\FF_{q^n}$. (Here, $j_1, \ldots, j_4$ are as in Proposition \ref{p:B2.10}.) Then under the same conditions as in Proposition \ref{p:B2.10}, we have
\begin{equation*}
H_c^i(S, P^* \Loc_\chi) \cong H_c^{i-2}(S_3, P_3^* \Loc_\chi)(-1)
\end{equation*}
as vector spaces equipped with an action of $\Fr_{q^n}$, where the Tate twist $(-1)$ means that the action of $\Fr_{q^n}$ on $H_c^{i-2}(S_3, P_3^* \Loc_\chi)$ is multiplied by $q^n$.
\end{proposition}

\begin{proof}
Let $P'(x,y) = g(f(x)^{q^j} y - f(x)^{q^n} y^{q^{n-j}})\cdot P_2(x).$ Then $P^* \Loc_\chi$ and $(P')^* \Loc_\chi$ are isomorphic since the pullback of $\Loc_\chi$ by the map $(0, \ldots, 0, z) \mapsto (0, \ldots, 0, z^{q^n})$ is trivial. Then by Proposition \ref{p:B2.10}, the desired conclusion holds.
\end{proof}

The following proposition is extremely useful in the context of applying the inductive argument described by the above propositions.

\begin{proposition}\label{p:claim1gen}
Suppose that $S \hookrightarrow R$ is a finite map of polynomial rings over $k = \overline \FF_q$. Assume that $\Frac R$ is finite Galois over $\Frac S$ with Galois group $G$ a $p$-group. Then
\begin{enumerate}[label=(\alph*)]
\item
$R$ is stable under $G$ and $R^G = S$

\item
the quotient of monoids $((R \smallsetminus \{0\})/k^\times)^G = (S \smallsetminus \{0\})/k^\times$

\item
If $(f) \subset R$ is an ideal such that $(\sigma f) = (f)$ for all $\sigma \in G$, then $f \in S$.
\end{enumerate}
\end{proposition}

\begin{proof}
First observe that since $S$ and $R$ are polynomial rings, they are normal and therefore integrally closed. Since $S \hookrightarrow R$ is a finite map, $R$ is the integral closure of $S$ in $\Frac R$. Thus $R$ is $G$-stable. It is clear that $S \subset R^G$ and that $R^G$ is integrally closed in $\Frac S$. But since $S$ is integrally closed, we necessarily have $S = R^G$. This proves (a).

To see (b), consider the short exact sequence
\begin{equation*}
1 \to k^\times \to \Frac R^\times \to \Frac R^\times/k^\times \to 1
\end{equation*}
and take $G$-invariants to get a long exact sequence 
\begin{equation*}
1 \to k^\times \to \Frac S^\times \to (\Frac R^\times/k^\times)^G \to H^1(G, k^\times) \to \cdots
\end{equation*}
Since $G$ acts trivially on $k^\times$, we have $H^1(G,k^\times) = \Hom(G, k^\times),$ which is trivial since $G$ is a $p$-group. Thus $(\Frac R^\times/k^\times)^G = \Frac S^\times/k^\times$ and $((R \smallsetminus \{0\})/k^\times)^G = (S \smallsetminus \{0\})/k^\times$.

Now we prove (c). If $f = 0$, then we are done, so for the rest of the proof we may assume $f \neq 0$. Necessarily $\sigma f = f$ up to a unit in $R$, and thus their images in the quotient $(R \smallsetminus \{0\})/k^\times$ are equal. Thus the image of $f$ is in $((R \smallsetminus \{0\})/k^\times)^G = (S \smallsetminus \{0\})/k^\times$, and so $f \in S$.
\end{proof}

\section{Representations of $\UnipF$} \label{s:reps}

Let $\cG$ be the set of irreducible representations of $\UnipkF$ whose central character has trivial $\Gal(L/K)$-stabilizer. Let $\cA$ denote the set of all characters of $H(\F)$ whose restriction to the center $Z(\UnipkF)$ of $\UnipkF$ has trivial $\Gal(L/K)$-stabilizer.

In this section, we show that $\cG$ can be parametrized by $\cA$ and explicitly describe such a parametrization. There are two main cases of behavior, depending on the parameters $n$, $h$, and $k$.

\begin{definition}
Given a triple of positive integers $(n,h,k)$ such that $h \geq k+1$, we say that:
\begin{enumerate}[label=\textbullet]
\item
$(n,h,k)$ is in \textit{Case 1} if $(n-1)(h-k)^+$ is even.

\item
$(n,h,k)$ is in \textit{Case 2} if $(n-1)(h-k)^+$ is odd.
\end{enumerate}
\end{definition}

Consider the following subgroups of $\Unip$:
\begin{align*}
H'(\F) &\colonequals \left\{\sum_{i=0}^{n-1} A_i \tau^i : \text{$A_{ij} = 0$ if $i > 0$ and $j \leq \frac{h-k}{2} - \frac{i}{n}$}\right\} \subset \UnipF \\
H^+(\F) &\colonequals \left\{\sum_{i=0}^{n-1} A_i \tau^i : \text{$A_{ij} = 0$ if $i > 0$ and $j < \frac{h-k}{2} - \frac{i}{n}$ and $A_{n/2,(h-k-1)/2} \in \FF_{q^{n/2}}$}\right\} \\
\end{align*}
We will also need the subgroups
\begin{align*}
H_0'(\F) &\colonequals \{\textstyle \sum_{i=0}^{n-1} A_i \tau^i \in H'(\F) : A_0 \in Z(\UnipkF)\}, \\
H_0^+(\F) &\colonequals \{\textstyle \sum_{i=0}^{n-1} A_i \tau^i \in H^+(\F) : A_0 \in Z(\UnipkF)\}.
\end{align*}
Let
\begin{align}\label{e:cI}
\cI &= \left\{(i,j) : i = 0, \, 1 \leq j \leq h-1\right\} \cup \left\{(i,j) : 1 \leq i \leq n-1, \, \frac{h-k}{2} - \frac{i}{n} < j < h-k\right\}.
\end{align}
Notice that
\begin{align}
\label{e:caseindex} [H^+(\F) : H'(\F)] &= 
\begin{cases}
1 & \text{if $(n,h,k)$ is in Case 1,} \\
q^{n/2} & \text{if $(n,h,k)$ is in Case 2.}\end{cases} \\
\label{e:fullindex} [\UnipF : H^+(\F)] &= q^{n(n-1)(h-k)^+/2}
\end{align}

For $\chi \in \cA$, define an extension $\chi^\sharp$ of $\chi$ to $H'(\F)$ by 
\begin{equation*}
\chi^\sharp(\textstyle \sum A_i \tau^i) \colonequals \chi(A_0).
\end{equation*}
Fix any extension $\widetilde \chi$ of $\chi^\sharp$ to $H^+(\F)$. Note that in Case 1, necessarily $\widetilde \chi = \chi^\sharp$. In Case 2, there are $q^{n/2}$ choices of $\widetilde \chi$ since $H^+(\F)$ is abelian.

\begin{lemma}\label{l:centralext}
If $\rho \in \cG$ has central character $\omega$, then the restriction of $\rho$ to $H_0'(\F)$ contains the character
\begin{equation*}
\omega^\sharp(A_0 + A_1 \tau + \cdots + A_{n-1} \tau^{n-1}) \colonequals \omega(A_0).
\end{equation*}
\end{lemma}

\begin{proof}
We prove this by induction on the subgroups
\begin{align*}
G_1 &\colonequals \{A_0 + V^{h-k-1} A_1 \tau^{n-1}\} \subset H_0'(\F), \\
G_2 &\colonequals \{A_0 + V^{h-k-1} A_2 \tau^{n-2} + V^{h-k-1} A_1 \tau^{n-1}\} \subset H_0'(\F), \ldots 
\end{align*}
Consider the extension $\omega_1$ of $\omega$ to $G_1$ defined as
\begin{equation*}
\omega_1 \from G_1 \to \overline \QQ_\ell^\times, \qquad A_0 + V^{h-k-1} A_1 \tau^{n-1} \mapsto \psi(A_0).
\end{equation*}
Then for any $g_1 = 1 + VB \tau \in G_1$ and $h = A_0 + V^{h-k-1} A \tau^{n-1} \in H_0'(\F)$, we have
\begin{align*}
{}^{g_1} \omega^\sharp(h) 
&= \omega^\sharp(A_0 + V^{h-1}(BA^q - AB^{q^{n-1}}))
= \omega^\sharp(A_0) \psi(BA^q - AB^{q^{n-1}}).
\end{align*}
Since $\psi$ has conductor $q^n$, every character $\F \to \overline \QQ_\ell^\times$ can be written as $A \mapsto \psi(BA^q - AB^{q^{n-1}})$ for some $B \in \F$. Thus the restriction of $\rho$ to $G_1$ contains $\omega_1$. Applying the above argument for each $G_i \subseteq H_0'(\F)$ inductively proves that the restriction of $\rho$ to $H_0'(\F)$ contains $\omega^\sharp$.

Suppose that we are in Case 2 and let $i = n/2$ and $j = (h-k-1)/2$. Let $\widetilde \omega$ be any extension of $\omega^\sharp$ to $H_0^+(\F)$. To prove that $\rho|_{H_0^+(\F)}$ contains $\widetilde \omega$, it is enough to prove that the orbit of $\widetilde \omega$ under $\UnipkF$ conjugacy contains every extension of $\omega^\sharp$ to $H_0^+(\F)$. Indeed, for $g = 1 + V^jB \tau^i \in \UnipkF$ and $h = \sum A_i \tau^i \in H_0^+(\F)$, we have
\begin{align*}
\widetilde \omega \Big((1 + V^j B \tau^i)&(A_0 + A_1 \tau + \cdots + A_{n-1} \tau^{n-1})((1 + V^{h-1}B^{q^i+1}) - V^jB \tau^i)\Big) \\
&= \widetilde \omega\Big((A_0 + V^{h-1}(A_i(B - B^{q^i}))) + A_i \tau^i\Big) \\
&= \widetilde \omega(A_0 + A_i \tau^i) \psi(A_i(B - B^{q^i})).
\end{align*}
Since $\psi$ has conductor $q^n$, the 
\begin{equation*}
\#\{A_i \mapsto \psi(A_i(B - B^{q^{n/2}}))\} = q^{n/2},
\end{equation*}
and this completes the proof.
\end{proof}

\begin{theorem}\label{t:irrepdesc}
For any $\chi \in \cA$, the representation $\rho_\chi \colonequals \Ind_{H^+(\F)}^{\UnipkF}(\widetilde \chi)$ is irreducible with dimension $q^{n(n-1)(h-k)^+/2}$. Moreover, $\cG = \{\rho_\chi : \chi \in \cA\}$.
\end{theorem}

\begin{proof}
The dimension follows from \eqref{e:fullindex}. To prove irreducibility, we use Mackey's criterion. First note that it is clear that $H'(\F)$ centralizes $\chi^\sharp$ and $H^+(\F)$ centralizes $\widetilde \chi$. We must show that these are exactly the centralizers of these characters.

Let $V \from \bW \to \bW$ denote the Verschiebung map. Consider $1 + V^j B \tau^i \in \UnipkF \smallsetminus H'(\F)$ with $i \neq n/2$. Then $(i', j') \colonequals (n-i, h-1-k-j) \in \cI$ and for any $A \in \F$,
\begin{align*}
\widetilde\chi\Big((1 + V^j B \tau^i)&(1 + V^{j'} A \tau^{i'})(1 - V^j B \tau^i + \cdots)\Big) \\
&= \widetilde\chi\Big((1 + V^j(B)V^{j'}(A)^{q^i} - V^{j'}(A)V^j(B)^{q^{i'}}) + V^{j'} A \tau^{i'}\Big) \\
&= \widetilde \chi\Big(1 + V^{j'}(A) \tau^{i'}\Big) \cdot \psi\Big(V^j(B)V^{j'}(A)^{q^i} - V^{j'}(A)V^j(B)^{q^{i'}}\Big).
\end{align*}
Since $\psi$ has conductor $q^n$, it follows that $1 + V^j B \tau^i$ does not centralize $\widetilde \chi$. Now consider $1 + V^j B \tau^i \in \UnipkF \smallsetminus H^+(\F)$ with $i = n/2$ and $j = (h-k-1)/2$ so that $B \in \F \smallsetminus \FF_{q^{n/2}}$. Then for any $A \in \FF_{q^{n/2}}$, 
\begin{align*}
\widetilde \chi \Big((1 + V^j B \tau^i)&(1 + V^j A \tau^i)((1 + V^{h-1}B^{q^i+1}) - V^j B \tau^i)\Big)\\
&= \widetilde \chi \Big((1 + V^{h-1}(A(B-B^{q^i})) + V^j A \tau^i\Big) \\
&= \widetilde \chi (1 + V^j A \tau^i) \cdot \psi(A(B - B^{q^i})).
\end{align*}
Again, since $\psi$ has conductor $q^n$, it follows that $1 + V^j B \tau^i$ does not centralize $\widetilde \chi$. This completes the proof.
\end{proof}

\section{Juggling sequences, Witt vectors, and the varieties $X_h$}\label{s:jugglingdesc}

We give a description of $X_h$ in terms of juggling sequences that will be crucial in understanding the cohomology groups $H_c^i(X_h, \overline \QQ_\ell)$. In this section, we also include some computational lemmas that will be used in the proof of Theorem \ref{t:cohomdesc}.

\subsection{Juggling sequences}

\begin{definition}
A \textit{juggling sequence} of \textit{period $n$} is a sequence $(j_1, \ldots, j_n)$ of nonnegative integers satisfying the following condition:
\begin{equation*}
\text{The integers $i + j_i$ are all distinct modulo $n$.}
\end{equation*}
The \textit{number of balls} of a juggling sequence is the \textit{average} of a juggling sequence, $\displaystyle\frac{1}{n} \sum_{i=1}^n j_i$.
\end{definition}

Throughout, all juggling sequences will be of a fixed period $n$. The following lemmas are straightforward.

\begin{lemma}[Properties of juggling sequences]\label{l:juggprops}
\mbox{}
\begin{enumerate}[label=(\alph*)]
\item
If $(j_1, \ldots, j_n)$ is a juggling sequence, there exists a unique permutation $\sigma \in S_n$ such that
\begin{equation*}
(j_1, \ldots, j_n) \equiv (\sigma(1)-1, \ldots, \sigma(n)-n) \mod n.
\end{equation*}
Given a juggling sequence $j$, we will denote the corresponding permutation by $\sigma_j$.

\item
Let $c = (1 2 \cdots n) \in S_n$ and let $j$ be a juggling sequence. Then $\sigma_{c \cdot j} = c^{-1}\sigma_j c$. In particular, the map $j \mapsto \sgn \sigma_j$ is invariant under cyclic permutations.
\end{enumerate}
\end{lemma}


\begin{lemma}\label{l:juggleshape}
Let $j$ be a juggling sequence of period $n$ with $r$ balls and let $e_i \in \ZZ^n$ denote the $n$-tuple with a $1$ in the $i$th coordinate and $0$'s elsewhere
\begin{enumerate}[label=(\alph*)]
\item
If $j$ has a coordinate labelled $rn$, then $j = (rn) \cdot e_1$ up to cyclic permutation.

\item
Let $s \leq rn$ be a positive integer with $n \nmid s$. Let $\bar s = s \pmod n$. If $j$ consists of coordinates labelled only by $0,$ $s$, and $rn-s$, then $j = s \cdot e_1 + (rn-s) \cdot e_{\bar s + 1}$ up to cyclic permutation.
\end{enumerate}
\end{lemma}

\subsection{$\cO_K$-Witt vectors}

The following lemmas are well-known.

\begin{lemma}\label{l:Zwitt}
The polynomials $S_r, M_r \in \cO_K[X_0, \ldots, X_r, Y_0, \ldots, Y_r]$ are
\begin{align*}
S_r(X,Y) &= X_r + Y_r + \sum_{i=1}^r \frac{1}{\pi^i} (X_{r-i}^{q^i} + Y_{r-i}^{q^i} - S_{r-i}(X,Y)^{q^i}), \\
M_r(X,Y) &= \sum_{i = 0}^r \pi^i\left(\sum_{j=0}^{r-i} X_{i+j}^{q^{r-i-j}} Y_{r-j}^{q^j}\right) + \sum_{i=1}^r \frac{1}{q^i}\left(\left(\sum_{j=0}^{r-i} X_j^{q^{r-j}} Y_{r-i-j}^{q^{i+j}}\right) - M_i(X,Y)^{q^i}\right).
\end{align*}
\end{lemma}

\begin{notation}
We write $\epsilon_r$ to mean any polynomial in $X_i^\alpha Y_j^\beta$ with $i + j < r$. We write $\delta_r$ to mean any polynomial whose monomials are products of indeterminants whose indices are $<r$. \hfill $\Diamond$
\end{notation}

\begin{lemma}\label{l:pwitt}
Let $A$ be an $\FF_p$-algebra. Then for $X,Y \in \bW(A)$,
\begin{align*}
(X +_\bW Y)_r &= X_r + Y_r + \epsilon_r, \\
(X \times_\bW Y)_r &= \sum_{j=0}^{r} X_{j}^{p^{r-j}} Y_{r-j}^{p^j} + \epsilon_r.
\end{align*}
\end{lemma}

\subsection{The varieties $X_h$}

We now use the above definitions together with some basic computational results about the ring of Witt vectors to describe the varieties $X_h \subset \Unipk$. We coordinatize $\Unipk = \Affine^{(h-1) + (n-1)(h-k)^+}$ by writing $A_0 + A_1 \tau + \cdots + A_{n-1} \tau^{n-1} \in \Unipk$ with $A_0 = (x_0, x_n, \ldots, x_{(h-1)n}) \in \bW_h$ and $A_i = (x_i, x_{i+n}, \ldots, x_{i+(h-k-1)n}) \in \bW_{h-k}$ for $i = 1, \ldots, n-1$. Let 
\begin{equation}
\sI \colonequals \{i : \text{$0 \leq i \leq (h-1)n$ if $n \mid i$, and $0 \leq i \leq (h-k-1)n$ if $n \nmid i$}\}.\label{e:sI}
\end{equation}
Given a juggling sequence $j \in \sI^n$, we have an associated permutation $\sigma \in S_n$. Let
\begin{equation*}
f_j \colonequals \#\{r : r > \sigma(r)\}
\end{equation*}
denote the number of anti-exceedances of $\sigma$.

\begin{lemma}\label{l:Xhpolys}
\begin{enumerate}[label=(\alph*)]
\item
In the equal characteristics case, the scheme $X_h \subset \Unip$ is defined by the vanishing of the polynomials
\begin{equation*}
g_{r} \colonequals \sum_j (-1)^{\sgn(\sigma_j)} x_{j_1}^q x_{j_2}^{q^2} \cdots x_{j_{n-1}}^{q^{n-1}}(x_{j_n}^{q^n} - x_{j_n}),
\end{equation*}
where $r=mn$, $x_0 \colonequals 1,$ and the sum ranges over juggling sequences $j = (j_1, \ldots, j_n) \in \sI^n$ with $|j| = \sum j_i = (m - (k-1) f_j)n$.

\item
In the mixed characteristics case, the scheme $X_h \subset \Unip$ is defined by the vanishing of the polynomials
\begin{equation*}
g_{r} \colonequals \sum_j (-1)^{\sgn \sigma_j} x_{j_1}^{q^{m- \lfloor j_1/n \rfloor}q} 
\cdots x_{j_{n-1}}^{p^{m- \lfloor j_{n-1}/n \rfloor} q^{n-1}}(x_{j_n}^{q^n} - x_{j_n})^{p^{m- \lfloor j_n/n \rfloor}} + \epsilon_{nm}.
\end{equation*}
where $r = mn$, $x_0 \colonequals 1,$ and the sum ranges over juggling sequences $j = (j_1, \ldots, j_n) \in \sI^n$ with $|j| = \sum j_i = r - (k-1) f_j n$.
\end{enumerate}
\end{lemma}

\begin{proof}[Proof of (a)]
Let $A$ denote the matrix associated to $A_0 + A_1 \tau + \cdots + A_{n-1} \tau^{n-1}$ and let $A_{r,s}$ denote the $(r,s)$th entry of $A$. Then if we set $x_i = 0$ for $i \notin \sI$,
\begin{equation*}
A_{r,s} = 
\begin{cases}
\sum x_{ni+s-r}^{q^{r-1}} \pi^i & \text{if $r \leq s$}, \\
\sum x_{n(i-k+1)+s-r}^{q^{r-1}} \pi^{i} & \text{if $r > s$}.
\end{cases}
\end{equation*}
Let $c_m$ denote the coefficient of $\pi^m$ in
\begin{equation*}
\det A = \sum_{\sigma \in S_n} (-1)^{\sgn \sigma} \prod_{r=1}^n A_{r,\sigma(r)}.
\end{equation*}
Then
\begin{equation*}
c_m = \sum_{\sigma \in S_n} (-1)^{\sgn \sigma} \sum_{|i| = m} \prod_{k=1}^n x_{ni_r^*+\sigma(r)-r}^{q^{k-1}},
\end{equation*}
where $i = (i_1, \ldots, i_n) \in \bZ_{\geq 0}^n$ and 
\begin{equation*}
i_r^* = \begin{cases} i_r & \text{if $r \leq \sigma(r)$,} \\ i_r - n(k-1) & \text{if $r > \sigma(r)$.} \end{cases}
\end{equation*}
Then setting $j_r \colonequals n i_r^* + \sigma(r) - r$ defines a juggling sequence $j = (j_1, \ldots, j_n) \in \sI^n$ with 
\begin{equation*}
|j| = \sum_{r=1}^n j_r = \sum_{r=1}^n n i_r^* + \sigma(r) - r = nm - n(k-1)f_j.
\end{equation*}
It is clear that every juggling sequence $j \in \sI^n$ arises in this way, and we therefore have
\begin{equation*}
c_m = \sum_j (-1)^{\sgn \sigma_j} x_{j_1} x_{j_2}^q \cdots x_{j_n}^{q^{n-1}},
\end{equation*}
where the sum ranges over juggling sequences $j \in \sI^n$ with $|j| = nm - n(k-1)f_j$.

Recall that $X_h$ is defined by the equations $c_m^q - c_m$ for $1 \leq m \leq h-1$. Let $c = (12\cdots n) \in S_n$ and let $j$ be any juggling sequence with $|j| = mn.$ By Lemma \ref{l:juggprops}, $c \cdot j$ is a juggling sequence $|c \cdot j| = mn$, and $\sgn (\sigma_{c \cdot j}) = \sgn(\sigma)$. Thus we may arrange the terms in $c_m^q - c_m$ so that we obtain: 
\begin{equation*}
c_m^q - c_m 
= \sum_j (-1)^{\sgn \sigma_j} x_{j_1}^q x_{j_2}^{q^2} \cdots x_{j_{n-1}}^{q^{n-1}}(x_{j_n}^{q^n} - x_{j_n}).
\end{equation*}
Writing $r = mn$ and letting $g_r \equalscolon c_{m}^q - c_m$ completes the proof of (a). The proof of (b) follows the same computation, modulo the multiplication rule $V^i(a) V^j(b) = V^{i+j}(a^{q^j} b^{q^i}) \in \bW_h$.
\end{proof}

\begin{corollary}
$X_h$ is a variety of pure dimension $(n-1)(h-k)^+$.
\end{corollary}

We will need the following computational lemmas in the proof of Theorem \ref{t:cohomdesc}.

\begin{lemma}\label{l:s(x)y coord}
Let $y = Y_0 + Y_1 \tau + \cdots + Y_{n-1} \tau^{n-1} \in H'(\overline \FF_q)$ and $s(x) = 1 + X_1 \tau + \cdots + X_{n-1} \tau^{n-1}$, where $X_i = (X_{ij})_{j=1, \ldots, h-k}$ and $X_{ij} = 0$ unless $(i,j) \in J$.
\begin{enumerate}[label=(\alph*)]
\item
In the equal characteristics case, if we write $s(x) \cdot y = \sum A_i \tau^i$, then
\begin{equation*}
A_i = Y_{ij} + \sum X_{i_1,j_1} Y_{i_1,j_1}^{q^{i_1+nj_1}},
\end{equation*}
where the sum varies over pairs $((i_1,j_1),(i_2,j_2)) \in J \times I$ such that 
\begin{equation*}
\begin{cases}
j_1 + j_2 = j & \text{if $i_1 + i_2 = i$,} \\
j_1 + j_2 = j - k & \text{if $i_1 + i_2 = i +n$.}
\end{cases}
\end{equation*}

\item
In the mixed characteristics case, if we write $s(x) \cdot y = 1 + \sum A_i \tau^i$, then
\begin{equation*}
A_i =
\begin{cases}
Y_{ij} + \sum (X_{i_1,j_1}^{q^{j_2}} Y_{i_1,j_1}^{q^{i_1 + j_1 + n j_1}})^q & \text{if $i = 0$}, \\
Y_{ij} + \sum X_{i_1,j_1}^{q^{j_2}} Y_{i_1,j_1}^{q^{i_1 + j_1 + n j_1}} & \text{if $i \neq 0$}
\end{cases}
\end{equation*}
where both sums vary over pairs $((i_1,j_1),(i_2,j_2)) \in J \times I$ such that 
\begin{equation*}
\begin{cases}
j_1 + j_2 = j & \text{if $i_1 + i_2 = i$,} \\
j_1 + j_2 = j - k & \text{if $i_1 + i_2 = i +n$.}
\end{cases}
\end{equation*}
\end{enumerate}
\end{lemma}

\begin{proof}
This is a straightforward computation.
\end{proof}

\begin{lemma}\label{l:claim1}
Let $g_k$ be as in Lemma \ref{l:Xhpolys} and let $s(x)$ be as in Lemma \ref{l:s(x)y coord}. Suppose that for any $y, y' \in H'(\overline \FF_q)$ with $L_{q^n}(y) = L_{q^n}(y')$,
\begin{equation*}
g_r(s(x) \cdot y) = 0 \qquad \Longleftrightarrow \qquad g_r(s(x) \cdot y') = 0.
\end{equation*}
Then using the identity $L_{q^n}(y) = 1 + \sum x_i \tau^i$, we have that $g_r(s(x) \cdot y)$ is a polynomial in $x_i$ for $i \in \sI$.
\end{lemma}


This is a corollary of Proposition \ref{p:claim1gen}.

\begin{proof}
For $i \in J$, define $y_i \colonequals x_i.$ Consider the rings
\begin{equation*}
R = \overline \FF_q[y_i : i \in \sI] \supset S = \overline \FF_q[x_i : i \in \sI]
\end{equation*}
and their fraction fields
\begin{equation*}
E = \Frac R = \overline \FF_q(y_i : i \in \sI) \supset F = \Frac S = \overline \FF_q(x_i : i \in \sI).
\end{equation*}
In order to apply Proposition \ref{p:claim1gen}, we need to show that $S \hookrightarrow R$ is a finite morphism. To do this, we just need to observe that for each $i \in \sI$, the polynomial $y_i$ is integral over $S$.

Notice that $E/F$ is not Galois but is the compositum of a tower of Galois extensions, where each nontrivial extension has Galois group $\FF_{q^n}$. Explicitly,
\begin{equation*}
E = E_0 \subseteq E_1 \subseteq \cdots \subseteq E_{n(h-1)} = F,
\end{equation*}
where $E_i = E_{i-1}(y_i).$ Then $E_i = \Frac S_i$ where $S_i = S_{i-1}[y_i]$.

By the set-up, we know that for each $\sigma \in \Gal(E_{n(h-1)}/E_{n(h-k)-1})$,
\begin{equation*}
g_r(s(x) \cdot y) = 0 \quad \Longleftrightarrow \quad g_r(s(x) \cdot \sigma(y)) = 0.
\end{equation*}
By the Nullstellensatz, this implies that the ideal generated by $g_r(s(x) \cdot y)$ in $S_{n(h-1)} = R$ is equal to the ideal generated by $g_r(s(x) \cdot \sigma(y))$. Thus by Proposition \ref{p:claim1gen}, we have $g_r(s(x) \cdot y) \in S_{n(h-k)-1}$. By induction, $g_r(s(x) \cdot y) \in S$.
\end{proof}

To prove Proposition \ref{p:dimHom}, we will actually need a more precise result than Lemma \ref{l:claim1}. For $* \in \bN$, let $\overline * \equiv * \pmod n$. 

\begin{lemma}\label{l:poly}
Let $a = 1 + \sum a_i \tau^i$ be as in Lemma \ref{l:s(x)y coord} and let $l \colonequals h-1,$ $m \colonequals h-k+1$, for ease of notation. Consider the descending sequence of integers $r_1 > r_2 > \cdots > r_s$ with $r_1 = mn - 1$, $r_s > n(h-k)/2$, and $n \nmid r_i$. Let $t_s(x) = x^{q^{r_s - n}} + x^{q^{r_s - 2n}} + \cdots + x^{q^{\bar r_s}}$. Set $x_{r_i} = 0 = x_{mn-r_i}$ for $i = 1, \ldots, s -1$.
\begin{enumerate}[label=(\alph*)]
\item
In the equal characteristic case,
\begin{equation*}
g_{ln}(a) = x_{ln} - x_{r_s}^{q^{n - \bar r_s}} x_{ln - r_s} + x_{r_s}^{q^n} x_{ln - r_s}^{q^{\bar r_s}} + x_{r_s}^{q^n} t_s(x_{ln - r_s})^{q^n} - x_{r_s} t_s(x_{ln - r_s}) + \delta_{ln-r_s}.
\end{equation*}

\item
In the mixed characteristic case,
\begin{align*}
g_{ln}(a) = x_{ln} &- x_{r_s}^{q^{n - \bar r_s + m - m_0}} x_{mn - r_s}^{q^{m_0+1}} + x_{r_s}^{q^{n+m-m_0}} x_{mn - r_s}^{q^{\bar r_s + m_0 + 1}} \\
&+ x_{r_s}^{q^{n+m-m_0}} t_s(x_{mn - r_s}^{q^{m_0+1}})^{q^n} - x_{r_s}^{q^{m-m_0}} t_s(x_{mn - r_s}^{q^{m_0+1}}) + \delta_{mn-r_s} + \epsilon_{mn},
\end{align*}
where $m_0 = \floor{r_s/n}$.
\end{enumerate}
\end{lemma}

\begin{proof}
Since $x_{r_i} = 0$ for $i = 1, \ldots, s-1$, any juggling sequence $j = (j_1, \ldots, j_n)$ wherein $y_{mn-r_s}$ contributes to $g_{ln}$ nontrivially must have the following criteria:
\begin{enumerate}[label=\textbullet]
\item
$j_n \neq 0$

\item
The numbers $r_i$ and $mn - r_i$ for $i = 1, \ldots, s-1$ do not appear in $j$.
\end{enumerate}
Since $y_{mn - r_s}$ only appears as the coefficient of $\pi^u \tau^v$ for $nu + v \geq mn - r_s$, it follows that the terms in $g_{ln}$ involving $y_{mn - r_s}$ occur exactly in the parts corresponding to the juggling sequences
\begin{align*}
&ln \cdot e_n && \longleftrightarrow && 1 \in S_n, \\
&r_s \cdot e_{n - \bar r_s} + (mn - r_s) \cdot e_n && \longleftrightarrow && (n-\bar r_s, n) \in S_n, \\
&(mn - r_s) \cdot e_{\bar r_s} + r_s \cdot e_n && \longleftrightarrow && (\bar r_s, n) \in S_n.
\end{align*}

We now handle the equal and mixed characteristics cases separately.
\begin{enumerate}[label=(\alph*)]
\item
In the equal characteristics case, the terms involving $y_{mn-r_s}$ occur in the expression
\begin{equation*}
(a_{ln}^{q^n} - a_{ln}) - a_{r_s}^{q^{n-\bar r_s}}(a_{mn-r_s}^{q^n} - a_{mn-r_s}) - a_{mn-r_s}^{q^{\bar r_s}}(a_{r_s}^{q^n} - a_{r_s}) + \delta_{mn-r_s}.
\end{equation*}
Thus by Lemma \ref{l:s(x)y coord}(a), we see that the only terms involving $y_{mn-r_s}$ are
\begin{equation*}
((x_{r_s} y_{mn-r_s}^{q^{r_s}})^{q^n} - x_{r_s} y_{mn-r_s}^{q^{r_s}}) 
- x_{r_s}^{q^{n-\bar r_s}}(y_{mn-r_s}^{q^n} - y_{mn-r_s}) 
- y_{mn-r_s}^{q^{\bar r_s}}(x_{r_s}^{q^n} - x_{r_s}).
\end{equation*}
This reduces to the expression
\begin{align*}
-&x_{r_s}^{q^{n - \bar r_s}} x_{mn-r_s} 
- x_{r_s} (y_{mn-r_s}^{q^{r_s}} - y_{mn-r_s}^{q^{\bar r_s}}) 
+ x_{r_s}^{q^n}(y_{mn-r_s}^{q^{r_s+n}} - y_{mn-r_s}^{q^{\bar r_s}}) \\
&= -x_{r_s}^{q^{n - \bar r_s}} x_{mn-r_s} 
+ x_{r_s}^{q^n}(y_{mn-r_s}^{q^{\bar r_s + n}} - y_{mn-r_s}^{\bar r_s}) \\
&\qquad\qquad\qquad- x_{r_s}(y_{mn-r_s}^{q^{r_s}} - y_{mn-r_s}^{q^{\bar r_s}}) 
+ x_{r_s}^{q^n}(y_{mn-r_s}^{q^{r_s}} - y_{mn-r_s}^{q^{\bar r_s}})^{q^n}.
\end{align*}
Finally, (a) follows once we recall the fact
\begin{equation*}
x_{mn-r_s} = y_{mn-r_s}^{q^n} - y_{mn-r_s} + \delta_{mn-r_s}.
\end{equation*}

\item
In the mixed characteristics case, the terms involving $y_{mn-r_s}$ occur in the expression
\begin{equation*}
(a_{ln}^{q^n} - a_{ln})
- a_{r_s}^{q^{n- \bar r_s} q^{m - \floor{r_s/n}}}(a_{mn-r_s}^{q^n} - a_{mn-r_s})^{q^{m - \floor{(mn-r_s)/n}}}
- a_{mn-r_s}^{q^{\bar r_s} q^{m - \floor{(mn-r_s)/n}}} (a_{r_s}^{q^n} - a_{r_s})^{q^{m - \floor{r_s/n}}}.
\end{equation*}
Let $m_0 = \floor{r_s/n}$. Notice that $\floor{(mn - r_s)/n} = m - m_0 - 1$. By Lemma \ref{l:s(x)y coord}(b), we see that the only terms involving $y_{mn - r_s}$ are
\begin{align*}
x_{r_s}^{q^{m-m_0+n}} y_{mn - r_s}^{q^{r_s + m_0 + 1 + n}}
- x_{r_s}^{q^{m-m_0}} y_{mn - r_s}^{q^{r_s + m_0 + 1}}
- x_{r_s}^{q^{n-\bar r_s +m-m_0}}&(y_{mn-r_s}^{q^n} - y_{mn-r_s})^{q^{m_0+1}} \\
- &y_{mn - r_s}^{q^{\bar r_s + m_0 + 1}}(x_{mn - r_s}^{q^n} - x_{mn - r_s})^{q^{m-m_0}}.
\end{align*}
This reduces to the expression
\begin{align*}
-x_{r_s}^{q^{n-\bar r_s + m - m_0}}&x_{mn - r_s}^{q^{m_0+1}}
+ x_{r_s}^{q^{n + m - m_0}} x_{mn - r_s}^{q^{\bar r_s}} \\
&- x_{r_s}^{q^{m-m_0}}t(y_{mn-r_s}^{q^{m_0+1}}) + x_{r_s}^{q^{n+m-m_0}} t(x_{mn-r_s}^{q^{m_0+1}})^{q^n} + \delta_{mn - r_s}.\qedhere
\end{align*}
\end{enumerate}
\end{proof}

\section{The representations $H_c^\bullet(X_h)[\chi]$}\label{s:cohomdesc}

In this section, we prove the irreducibility of $H_c^i(X_h, \overline \QQ_\ell)[\chi]$ and its vanishing outside a single degree. The key proposition is:

\begin{proposition}\label{p:dimHom}
For any $\chi \in \cA$,
\begin{equation*}
\dim\Hom_{\UnipkF}(\rho_\chi, H_c^i(X_h, \overline \QQ_\ell)) = \delta_{i, (n-1)(h-k)^+},
\end{equation*}
where $\rho_\chi \in \cG$ is the representation described in Theorem \ref{t:irrepdesc}. Moreover, $\Fr_{q^n}$ acts on $H_c^{(n-1)(h-k)^+}(X_h, \overline \QQ_\ell)$ via multiplication by $(-1)^{(n-1)(h-k)^+} q^{n(n-1)(h-k)^+/2}.$
\end{proposition}

Recall that $\F^\times \ltimes \UnipF \cong \cR_{h,k,n,q}^\times(\F)$ and that $\F$ acts on $X_h$ by conjugation. For any $z \in \F$ and any $g,h \in H(\F)$, let $(z,h,g)$ denote the map $X_h \to X_h$ given by $x \mapsto z(h * x \cdot g)z^{-1}$. We prove the following proposition in Section \ref{s:pftrace}.

\begin{proposition}\label{p:vregtrace}
For any generator $\zeta$ of $\F^\times$,
\begin{equation*}
\Tr((\zeta, 1, g)^* ; H_c^{(n-1)(h-k)^+}(X_h, \overline \QQ_\ell)[\chi]) = (-1)^{(n-1)(h-k)^+} \chi(g).
\end{equation*}
\end{proposition}

From the multiplicity-one statement of Proposition \ref{p:dimHom}, the nonvanishing statement of Proposition \ref{p:vregtrace}, and a counting argument coming from Theorem \ref{t:irrepdesc}, one obtains the following theorem.

\begin{theorem}\label{t:cohomdesc}
For any $\chi \in \cA$, the $\UnipF$-representation $H_c^i(X_h, \overline \QQ_\ell)[\chi]$ is irreducible when $i = (n-1)(h-k)^+$ and vanishes otherwise. Moreover, for $\chi, \chi' \in \cA$, we have $H_c^{(n-1)(h-k)^+}(X_h, \overline \QQ_\ell)[\chi] \cong H_c^{(n-1)(h-k)^+}(X_h, \overline \QQ_\ell)[\chi']$ if and only if $\chi = \chi'$.
\end{theorem}

We prove this in Section \ref{s:pfcohomdesc}. It is a trivial corollary of Theorem \ref{t:cohomdesc} that

\begin{corollary}\label{c:Gbij}
The parametrization
\begin{equation*}
\cA \to \cG, \qquad \chi \mapsto \rho_\chi
\end{equation*}
described in Theorem \ref{t:irrepdesc} is a bijection.
\end{corollary}

\subsection{Proof of Proposition \ref{p:dimHom}}\label{s:pfdimHom}

Note that from Section \ref{s:reps}, the representation
\begin{equation*}
W_\chi \colonequals \Ind_{H'(\F)}^{\UnipkF}(\chi^\sharp)
\end{equation*}
is irreducible and isomorphic to $\rho_\chi$ in Case 1, and is a direct sum of $q^{n/2}$ copies of $\rho_\chi$ in Case 2. Thus the statement of the proposition is equivalent to the following two statements:
\begin{enumerate}[label=(\alph*)]
\item
If we are in Case 1, then
\begin{equation*}
\dim \Hom_{\UnipkF}(W_\chi, H_c^i(X_h, \overline \QQ_\ell)) = \delta_{i,(n-1)(h-k)^+}.
\end{equation*}

\item
If we are in Case 2, then
\begin{equation*}
\dim \Hom_{\UnipkF}(W_\chi, H_c^i(X_h, \overline \QQ_\ell)) = q^{n/2} \cdot \delta_{i,(n-1)(h-k)^+}.
\end{equation*}
\end{enumerate}

We use Proposition \ref{p:B2.3} to reduce the computation of the space $\Hom_{\UnipF}(W_\chi, H_c^i(X_h, \overline \QQ_\ell))$ to a computation of the cohomology of a certain scheme $S$ with coefficients in a certain constructible $\QQ_\ell$-sheaf $\sF$. Then, to compute $H_c^i(S, \sF)$, we inductively apply Proposition \ref{p:B2.10}, a more general version of Proposition 2.10 in \cite{B12}. This will allow us to reduce the computation to a computation involving a $0$-dimensional scheme in Case 1 and a $1$-dimensional scheme in Case 2. We will treat these cases simultaneously until the final step.

\subsubsection*{Step 0.}

We first need to establish some notation.

\begin{enumerate}[label=\textbullet]
\item 
Let $I \colonequals \{nj + i : (i,j) \in \cI\}$, where $\cI$ was defined in Equation \eqref{e:cI}. Put $I' = I \smallsetminus \{n, 2n, \ldots, n(h-1)\}$ and let $d = |I'| = \floor{(n-1)(h-1)/2}.$ Put $J \colonequals \sI \smallsetminus I$.

\item
Let $r_s$ denote the $s$th smallest number in $J$ not divisible by $n$. Put $I_0 \colonequals I'$ and $J_0 \colonequals J$. Then define $I_s \colonequals I_{s-1} \smallsetminus \{n(h-k)-r_s\}$ and $J_s \colonequals J_{s-1} \smallsetminus \{r_s\}$.

\item
Note that $I_d = \varnothing$. In Case 1, $J_d = \varnothing.$ In case 2, $J_d = \{n(h-1)/2\}$.

\item
Note that $H' = \{1 + \sum a_i \tau^i : i \in I\}$.

\item
For a finite set $A \subset \NN$, we will write $\Affine [A]$ to denote the affine space of dimension $|A|$ with coordinates labelled by $A$.

\item
For $* \in \NN$, we will denote by $[*]$ the representative of $*$ in $\{1, \ldots, n\}$ modulo $n$. We will write $\bar * \equiv * \pmod n$. Note that these only differ when $* \equiv 0 \pmod n$.
\end{enumerate}

\subsubsection*{Step 1}

We apply Proposition \ref{p:B2.3} to the following set-up:
\begin{enumerate}[label=\textbullet]
\item
the group $\Unipk$ together with the connected subgroup $H'$, both of which are defined over $\F$

\item
a morphism $s \from \Unipk/H' \to \Unipk$ defined by identifying $\Unipk/H'$ with affine space $\Affine[J]$ and setting $s \from (x_{nj+i})_{nj+i \in J} \mapsto 1 + \sum_{nj+i \in J} V^j(x_{nj+i}) \tau^i$

\item
the algebraic group morphism $f \from H' \to H$ given by $\sum A_i \tau^i \mapsto A_0$

\item
an additive character $\chi \from H(\F) \to \overline \QQ_\ell^\times$

\item
a locally closed subvariety $Y_h \subset \Unipk$ which is chosen so that $X_h = L_{q^n}^{-1}(Y_h)$
\end{enumerate}
Since $X_h$ has a right-multiplication action of $\UnipkF$, the cohomology groups $H_c^i(X_h, \overline \QQ_\ell)$ inherits a $\UnipkF$-action. For each $i \geq 0$, Proposition \ref{p:B2.3} implies that we have a vector space isomorphism
\begin{equation*}
\boxed{\Hom_{\UnipF}(W_\chi, H_c^i(X_h, \overline \QQ_\ell)) \cong H_c^i(\beta^{-1}(Y_h), P^* \Loc_\chi)}
\end{equation*}
compatible with the action of $\Fr_{q^n}$. Here, $\Loc_\chi$ is the local system on $H$ corresponding to $\chi$, the morphism $\beta \from (\Unip/H') \times H' \to \Unip$ is given by $\beta(x,g) = s(\Fr_{q^n}(x)) \cdot g \cdot s(x)^{-1}$, and the morphism $P \from \beta^{-1}(Y_h) \to H$ is the composition $\beta^{-1}(Y_h) \hookrightarrow (\Unip/H') \times H' \overset{\pr}{\longrightarrow} H' \overset{f}{\longrightarrow} H.$

We now work out an explicit description of $\beta^{-1}(Y_h) \subset \Affine[J] \times H'$. For $1 \leq k \leq n(h-1)$ and $k$ divisible by $n$, let $g_k$ be the polynomial described in Lemma \ref{l:Xhpolys}. Write $x = (x_i)_{i \in J} \in \Affine[J]$ and $g = 1 + \sum_{i \in I} x_i \tau^i \in H'(\overline \FF_q)$. Now pick $y = 1 + \sum_{i \in I} y_i \tau^i \in H'(\overline \FF_q)$ such that $L_{q^n}(y) = g$. Then
\begin{equation*}
\beta(x,g) = \Fr_{q^n}(s(x)) \cdot L_{q^n}(y) \cdot s(x)^{-1} = L_{q^n}(s(x) \cdot y).
\end{equation*}
We see that $\beta(x,g) \in Y_h$ if and only if $s(x) \cdot y \in X_h$. Let $s(x) \cdot y = 1 + \sum a_i \tau^i$. By Lemma \ref{l:Xhpolys}, we know that $s(x) \cdot y \in X_h$ if and only if $g_k(a) = 0$ for all $k \leq n(h-1)$ divisible by $n$. Recall from Lemma \ref{l:claim1} that using the identity $L_{q^n}(y) = 1 + \sum_{i \in I} x_i \tau^i$, each polynomial $g_k(a)$, which \textit{a priori} is a polynomial in $x_j$ for $j \in J$ and $y_i$ for $i \in I$, is in fact a polynomial in $x_i$ for $1 \leq i \leq n(h-1)$.

\subsubsection*{Step 2}

By Lemma \ref{l:claim1}, we know that each polynomial $g_r(s(x) \cdot y)$ can be written as a polynomial in $x_i$'s. Furthermore, they can each be written as a polynomial in $x_i$ for $i \in I_0 \cup J_0$, since $g_r(a)$ is of the form $x_r + \text{(stuff with $x_i$, $i \leq r$)}$. Thus the $in$th coordinates $x_{in}$ of $\beta^{-1}(Y_h) \subset \Affine[I \cup J]$ are uniquely determined by the other coordinates, which implies that for a subscheme $S^{(0)} \subset \Affine[I_0 \cup J_0]$,
\begin{equation*}
H_c^i(\beta^{-1}(Y_h), P^* \Loc_\chi) \cong H_c^i(S^{(0)}, (P^{(0)})^* \Loc_\chi),
\end{equation*}
where $P^{(0)} \from S^{(0)} \to Z$ is the map given by $(x_i)_{i \in I_0 \cup J_0} \mapsto (x_n, x_{2n}, \ldots, x_{(h-1)n})$.

We now apply Proposition \ref{p:B2.10} to the following set-up:
\begin{enumerate}[label=\textbullet]
\item
Let $S^{(0)} = \Affine[I_0 \cup J_0]$.

\item
Let $S_2^{(0)} = \Affine[I_1 \cup J_0]$.

\item
Note that $S^{(0)} = S_2^{(0)} \times \Affine[\{n(h-1)-1\}]$.

\item
Let $f \from S_2^{(0)} \to \GG_a$ be defined as $(x_i)_{i \in I_1 \cup J_0} \mapsto x_1$.

\item
Set $v \in S_2^{(0)}$ and $w = x_{n(h-k)-1}$. By Lemma \ref{l:poly}, in the equal characteristics case, we may write
\begin{equation*}
P^{(0)}(v,w) = g(f(v)^{q^{n-1}} w - f(v)^{q^n} w^q) \cdot P_2^{(0)}(v),
\end{equation*}
and in the mixed characteristics case, we may write
\begin{equation*}
P^{(0)}(v,w) = g(f(v)^{q^{n}} w^{q^{h-k}} - f(v)^{q^{n+1}} w^{q^h}) \cdot P_2^{(0)}(v).
\end{equation*}

\item
Let $S_3^{(0)} = \Affine[I_1 \cup J_1]$ so that this is the subscheme of $S_2^{(0)} = \Affine[I_1 \cup J_0]$ defined by $f = 0$, and let $P_3^{(0)} \colonequals P_2^{(0)}|_{S_3^{(0)}} \from S_3^{(0)} \to Z$.
\end{enumerate}
Then by Proposition \ref{p:B2.10}, for all $i \in \ZZ$,
\begin{equation*}
\boxed{H_c^i(S^{(0)},(P^{(0)})^* \Loc_\chi) \cong H_c^{i-2}(S_3^{(0)}, (P_3^{(0)})^* \Loc_\chi)(-1)}
\end{equation*}
as vector spaces equipped with an action of $\Fr_q$, where the Tate twist $(-1)$ means that the action of $\Fr_q$ on $H_c^{i-2}(S_3^{(0)}, (P_3^{(0)})^* \Loc_\chi)$ is multiplied by $q$. Note that this implies that the action of $\Fr_{q^n}$ is multiplied by $q^n$.

\subsubsection*{Step 3}

We now describe the inductive step for $l < d$. We apply Proposition \ref{p:B2.10alpha} to the following set-up:
\begin{enumerate}[label=\textbullet]
\item
Let $S^{(l)} \colonequals S_3^{(l-1)} = \Affine[I_l \cup J_l].$

\item
Let $S_2^{(l)} = \Affine[I_{l+1} \cup J_l]$.

\item
Note that $S^{(l)} = S_2^{(l)} \times \Affine[\{n(h-1)-r_l\}]$.

\item
Let $f \from S_2^{(l)} \to Z$ be defined as $(x_i)_{i \in I_{l+1} \cup J_l} \mapsto (0, \ldots, 0, x_{r_l})$.

\item
Set $v \in S_2^{(l)}$ and $w = x_{n(h-k)-r_l}$. Let $t_k(x) = x^{q^{r_l-n}} + \cdots + x^{q^{\bar r_l}}.$ By Lemma \ref{l:poly}, the morphism $P^{(l)} \colonequals P_3^{(l-1)} \from S^{(l)} \to Z$ has the following form: In the equal characteristics case,
\begin{equation*}
P^{(l)}(v,w) = g(f(v)^{n - \bar r_l} w - f(v)^{q^n} w^{q^{\bar r_l}}) \cdot P_2^{(l)}(v),
\end{equation*}
and in the mixed characteristics case,
\begin{equation*}
P^{(l)}(v,w) = g(f(v)^{q^{n- \bar r_l + h - k - m}} w^{q^{m+1}} - f(v)^{q^{n + h - k - m}} w^{\bar r_l + m + 1}) \cdot P_2^{(l)}(v).
\end{equation*}

\item
Let $S_3^{(l)} = \Affine[I_{l+1} \cup J_{l+1}]$ so that this is the subscheme of $S_2^{(l)} = \Affine[I_{l+1} \cup J_l]$ defined by $f = 0$, and let $P_3^{(l)} \colonequals P_2^{(l)}|_{S_3^{(l)}} \from S_3^{(l)} \to Z$.
\end{enumerate}
Then by Proposition \ref{p:B2.10}, for all $i \in \ZZ$,
\begin{equation*}
\boxed{H_c^i(S^{(l)}, (P^{(l)})^* \Loc_\chi) \cong H_c^{i-2}(S_3^{(l)}, (P_3^{(l)})^* \Loc_\chi)(-1).}
\end{equation*}

\subsubsection*{Step 4: Case 1}

Step 3 allows us to reduce the computation about the cohomology of $S^{(0)}$ to a computation about the cohomology of $S^{(d)} \colonequals S_3^{(d-1)}$, which is a point. Thus $\Fr_{q^n}$ acts trivially on the cohomology of $S^{(d)}$ and
\begin{equation*}
\boxed{\dim H_c^i(S^{(d)}, (P^{(d)})^*\Loc_\chi) = \delta_{0, i}.}
\end{equation*}

\subsubsection*{Step 4: Case 2}

Step 3 allows us to reduce the computation about the cohomology of $S^{(0)}$ to a computation about the cohomology of $S^{(d)} \colonequals S_3^{(d-1)} = \Affine[\{n(h-k)/2\}]$. The morphism $P^{(d)}$ is 
\begin{equation*}
P^{(d)} \from S^{(d)} \to Z, \qquad a_{n(h-k)/2} \mapsto \left(0, \ldots, 0, a_{n(h-k)/2}^{q^{n/2}}(a_{n(h-k)/2}^{q^n} - a_{n(h-k)/2})\right).
\end{equation*}
Then I claim that
\begin{equation*}
H_c^i(\GG_a, (P^{(d)})^* \Loc_\chi) = H_c^i(\GG_a, P^* \Loc_\psi),
\end{equation*}
where $\psi$ is the restriction of $\chi$ to $\F \to \overline \QQ_\ell^\times$ and $P_0$ is the morphism
\begin{equation*}
P_0 \from \GG_a \to \GG_a, \qquad x \mapsto x^{q^{n/2}}(x^{q^n} - x).
\end{equation*}

We now compute the cohomology groups $H_c^i(\GG_a, P^* \Loc_\psi)$ in the same way as in sections 6.5 and 6.6 in \cite{BW14}. We may write $P = f_1 \circ f_2$ where $f_1(x) = x^{q^{n/2}} - x$ and $f_2(x) = x^{q^{n/2}+1}$. Since $f_1$ is a group homomorphism, then $f_1^* \Loc_\psi \cong \Loc_{\psi \circ f_1}$. By assumption $\psi$ has trivial $\Gal(\F/\FF_q)$-stabilizer, so $\psi \circ f_1$ is nontrivial. Furthermore, $\psi \circ f_1$ is trivial on $\FF_{q^{n/2}}$. Thus the character $\psi \circ f_1 \from \F \to \overline \QQ_\ell^\times$ satisfies the hypotheses of Proposition 6.12 of \cite{BW14}, and thus $\Fr_{q^n}$ acts on $H_c^1(\GG_a, P_0^* \Loc_\psi)$ via multiplication by $-q^{n/2}$ and
\begin{equation*}
\dim H_c^i(\GG_a, P_0^* \Loc_\psi) = q^{n/2} \cdot \delta_{1,i}.
\end{equation*}
Thus
\begin{equation*}
\boxed{\dim H_c^i(S^{(d)}, (P^{(d)})^* \Loc_\chi) = q^{n/2} \cdot \delta_{1,i}.}
\end{equation*}

\subsubsection*{Step 5}

We now put together all of the boxed equations. We have
\begin{align*}
\Hom_{\UnipF}(W_\chi, H_c^i(X_h, \overline \QQ_\ell))
& \cong H_c^i(\beta^{-1}(Y_h), P^* \Loc_\chi) \\
&= H_c^i(S^{(0)}, (P^{(0)})^* \Loc_\chi) \\
&\cong H_c^{i-2}(S_3^{(0)}, (P_3^{(0)})^* \Loc_\chi)(-1) \\
&= H_c^{i-2}(S^{(1)}, (P^{(1)})^* \Loc_\chi)(-1) \\
&\cong H_c^{i-2d}(S^{(d)}, (P^{(d)})^* \Loc_\chi)(-d).
\end{align*}
Therefore if we are in Case 1, then
\begin{equation*}
\dim \Hom_{H(\F)}(W_\chi, H_c^i(X_h, \overline \QQ_\ell)) = \delta_{(n-1)(h-k)^+,i}.
\end{equation*}
Moreover, the Frobenius $\Fr_{q^n}$ acts on $\Hom_{\UnipF}(W_\chi, H_c^i(X_h, \overline \QQ_\ell)_)$ via multiplication by the scalar $q^{n(n-1)(h-k)^+/2}$.

If we are in Case 2, then
\begin{equation*}
\dim \Hom_{\UnipF}(W_\chi, H_c^i(X_h, \overline \QQ_\ell)) = q^{n/2} \cdot \delta_{(n-1)(h-k)^+,i}.
\end{equation*}
Moreover, the Frobenius $\Fr_{q^n}$ acts on $\Hom_{\UnipF}(W_\chi, H_c^i(X_h, \overline \QQ_\ell))$ via multiplication by the scalar $-q^{n(n-1)(h-k)^+/2}$.

Finally, observe that if we are in Case 1, then $(n-1)(h-k)^+$ is even and if we are in Case 2, then $(n-1)(h-k)^+$ is odd. This gives us a uniform way to describe the action of $\Fr_{q^n}$ and concludes the proof of Proposition \ref{p:dimHom}.

\subsection{Proof of Proposition \ref{p:vregtrace}} \label{s:pftrace}

By the Deligne--Lusztig fixed point formula,
\begin{equation*}
\sum_i (-1)^i \Tr((\zeta, h, g)^* ; H_c^i(X_h, \overline \QQ_\ell)) = \sum_i (-1)^i \Tr((1,h,g)^* ; H_c^i(X_h^\zeta, \overline \QQ_\ell)).
\end{equation*}
It is easy to calculate $X_h^\zeta$. Indeed, it can be identified with the subvariety of all elements of $\Unipk$ of the form $A_0 \in 1 + \bW_{h-1}(\F)$. Thus $X_h^{\zeta}$ is just a discrete set naturally identified with $H(\F)$ and the left and right actions of $H(\F)$ are given by left and right multiplication. Therefore $H_c^i(X_h^\zeta, \overline \QQ_\ell) = 0$ for $i > 0$ so
\begin{equation*}
\sum_i (-1)^i \Tr((1,h,g)^* ; H_c^i(X_h^\zeta, \overline \QQ_\ell)) = \Tr((1,h,g)^* ; H_c^0(X_h^\zeta, \overline \QQ_\ell)).
\end{equation*}
Furthermore, as a $(H(\F) \times H(\F))$-representation, $H_c^0(X_h^\zeta, \overline \QQ_\ell)$ is the pullback of the regular representation of $H(\F)$ along the multiplication map $H(\F) \times H(\F) \to H(\F)$. Thus
\begin{equation*}
H_c^0(X_h^\zeta, \overline \QQ_\ell) = \bigoplus_{\chi_0 \in \widehat H(\F)} \chi_0 \otimes \chi_0
\end{equation*}
as representations of $H(\F) \times H(\F)$. Therefore
\begin{equation*}
\sum_{h \in H(\F)} \chi(h)^{-1} \sum_i (-1)^i \Tr((\zeta, h, g)^* ; H_c^i(X_h, \overline \QQ_\ell)) = \chi(g) \cdot \# H(\F).
\end{equation*}
This is equivalent to
\begin{equation*}
\sum_i (-1)^i \Tr((\zeta, 1, g)^* ; H_c^i(X_h, \overline \QQ_\ell)[\chi]) = \chi(g),
\end{equation*}
and since $H_c^i(X_h, \overline \QQ_\ell)[\chi] = 0$ for $i \neq (n-1)(h-k)^+$ by Proposition \ref{p:dimHom}, the desired result follows.

\subsection{Proof of Theorem \ref{t:cohomdesc}} \label{s:pfcohomdesc}

This is a corollary of Proposition \ref{p:dimHom} and \ref{p:vregtrace}. We have
\begin{equation*}
\bigoplus_{\chi \in \cA} H_c^{(n-1)(h-k)^+}(X_h, \overline \QQ_\ell)[\chi] = \bigoplus_{\chi \in \cA} \rho_\chi,
\end{equation*}
where the summands on the left-hand side are mutually nonisomorphic and the summands on the right-hand side are irreducible. It follows then that each $H_c^{(n-1)(h-k)^+}(X_h, \overline \QQ_\ell)[\chi]$ is irreducible.

\section{Division algebras and Jacquet--Langlands transfers} \label{s:divalg}

Our goal in this final section is to understand two connections. The first, explained in Section \ref{s:DL}, is to unravel the relationship between Theorem \ref{t:cohomdesc} and the representations of division algebras arising from $p$-adic Deligne--Lusztig constructions. Because the equal characteristics claim of Theorem \ref{t:cohomdesc} proves a conjecture of Boyarchenko (see Conjecture 5.18 of \cite{B12}), we can use Proposition 5.19 of \textit{op.\ cit.} to explicitly describe this relationship. In fact, the definitions in Section \ref{s:definitions} allow us to treat all characteristics simultaneously and extend Boyarchenko's work.

The second connection, explained in Section \ref{s:LLC}, is to unravel the relationship between the representations described in Section \ref{s:DL} with respect to the local Langlands and Jacquet--Langlands correspondences. The main theorem of this section is Theorem \ref{t:divalg}, which says, colloquially, that the correspondence $\theta \mapsto H_\bullet(\widetilde X)[\theta]$ is consistent with the correspondence given by the composition of the local Langlands and Jacquet--Langlands correspondences.

\subsection{Deligne--Lusztig constructions for division algebras} \label{s:DL}

Throughout this section, $\theta \from L^\times \to \overline \QQ_\ell^\times$ will be a primitive character of level $h$ and $\chi \from U_L^1/U_L^h \to \overline \QQ_\ell^\times$ will be the induced homomorphism.

Let $\Khat$ be the completion of the maximal unramified extension of $K$ and let $\varphi$ denote the Frobenius automorphism of $\Khat$ (inducing $x \mapsto x^q$ on the residue field). We can write $D \colonequals D = L \langle \Pi \rangle/(\Pi^n - \pi^k)$, where $L \langle \Pi \rangle$ is the twisted polynomial ring defined by the commutation relation $\Pi \cdot a = \varphi(a) \cdot \Pi$. Write $\cO_D = \cO_L \langle \Pi \rangle / (\Pi^n - \pi^k)$ for the ring of integers of $D$. Define $P_D^r = \Pi^r \cO_D$ and $U_D^r = 1 + P_D^r$.

There exists a connected reductive group $\GG$ over $K$ such that $\GG(K)$ is isomorphic to $D^\times$, and a $K$-rational maximal torus $\TT \subset \GG$ such that $\TT(K)$ is isomorphic to $L^\times$. More explicitly, the homomorphism
\begin{equation*}
F \from \GL_n(\Khat) \to \GL_n(\Khat), \qquad A \mapsto \varpi^{-1} A^\varphi \varpi, \qquad \text{where $\varpi = \left(\begin{smallmatrix} 
0 & 1 & 0 & \cdots & 0 \\
0 & 0 & 1 & \cdots & 0 \\
\vdots & \ddots & \ddots & \ddots & \vdots \\
0 & \cdots & 0 & 0 & 1 \\
\pi^k & 0 & 0 & 0 & 0
\end{smallmatrix}\right)$}
\end{equation*}
is a Frobenius relative to a $K$-rational structure whose corresponding algebraic group over $K$ is $\GG$.

Let $\widetilde G \colonequals \GG(\Khat) = \GL_n(\Khat)$ and $\widetilde T \colonequals \TT(\Khat)$. Let $\BB \subset \GG \otimes_K \Khat$ be the Borel subgroup consisting of upper triangular matrices and let $\UU$ be its unipotent radical. Note that $\widetilde T$ consists of all diagonal matrices and $\widetilde U \colonequals \UU(\Khat)$ consists of unipotent upper triangular matrices. Let $\widetilde U^- \subset \GL_n(\Khat)$ denote the subgroup consisting of unipotent lower triangular matrices.

The $p$-adic Deligne--Lusztig construction $X$ for $D^\times$ described in \cite{L79} is the quotient
\begin{equation*}
X \colonequals (\widetilde U \cap F^{-1}(\widetilde U)) \backslash \{A \in \GL_n(\Khat) : F(A)A^{-1} \in \widetilde U\},
\end{equation*}
which can be identified with the set
\begin{equation*}
\widetilde X \colonequals \{A \in \GL_n(\Khat) : F(A)A^{-1} \in \widetilde U \cap F(\widetilde U^-)\}.
\end{equation*}
Recall from Section \ref{s:Xhdef} that 
\begin{equation*}
\widetilde X = \bigsqcup_{m \in \ZZ} \varprojlim_h \widetilde X_h^{(m)},
\end{equation*}
where each $X_h^{(m)}$ is a scheme of finite type over $\overline \FF_q$. Following \cite{B12}, Section 4.4, we set
\begin{equation*}
H_i(\widetilde X, \overline \QQ_\ell) = \bigoplus_{m \in \ZZ} \varinjlim_h H_i(\widetilde X_h^{(m)}, \overline \QQ_\ell),
\end{equation*}
where $H_i(S, \overline \QQ_\ell) \colonequals H_c^{2d-i}(S, \overline \QQ_\ell(d))$ for any smooth $\overline \FF_q$-scheme $S$ of pure dimension $d$. For each $i \geq 0$, $H_i(\widetilde X, \overline \QQ_\ell)$ inherits commuting smooth actions of $\GG(K) \cong D^\times$ and $\TT(K) \cong L^\times$. Given a smooth character $\theta \from L^\times \to \overline \QQ_\ell^\times$, we may consider the subspace $H_i(\widetilde X, \overline \QQ_\ell)[\theta] \subset H_i(\widetilde X, \overline \QQ_\ell)$ wherein $L^\times$ acts by $\theta$.

\begin{proposition}\label{p:divalg}
Let $U_D^{(h)} \colonequals (1 + P_L^h)(1 + P_L^{(h-k)^+}\Pi) \cdots (1 + P_L^{(h-k)^+}\Pi^{n-1})$.
\begin{enumerate}[label=(\alph*)]
\item
The representation $H_c^{(n-1)(h-k)^+}(X_h, \overline \QQ_\ell)[\chi]$ extends uniquely to a representation of the semidirect product $\cR_{h,k,n,q}^\times(\F) \cong \cO_D^\times/U_D^{(h)}$ with $\Tr(\eta_\theta^\circ(\zeta)) = (-1)^{(n-1)(h-k)^+} \theta(\zeta).$

\item
The inflation $\widetilde \eta_\theta^\circ$ of $\eta_\theta^\circ$ to $\cO_D^\times$ extends to a representation $\eta_\theta'$ of $\pi^\ZZ \cdot \cO_D^\times$ by setting $\eta_\theta'(\pi) = \theta(\pi)$. Then
\begin{equation*}
H_{(n-1)(h-k)^+}(\widetilde X, \overline \QQ_\ell)[\theta] \cong 
\eta_\theta \colonequals \Ind_{\pi^\ZZ \cdot \cO_D^\times}^{D^\times}(\eta_\theta')
\end{equation*}
and $H_i(\widetilde X, \overline \QQ_\ell)[\theta] = 0$ for $i \neq (n-1)(h-k)^+$.

\item
$H_{(n-1)(h-k)^+}(\widetilde X, \overline \QQ_\ell)[\theta]$ is an irreducible representation of dimension $n \cdot q^{n(n-1)(h-k)^+/2}$.
\end{enumerate}
\end{proposition}

\begin{proof}
This is proved in the equal characteristics case for $k/n = 1/n$ in Section 6.15 of \cite{B12}. The proof generalizes without complications, and we outline the arguments here.

The uniqueness in (a) follows from the irreducibility of $H_c^{(n-1)(h-k)^+}(X_h, \overline \QQ_\ell)[\chi]$. The representation $\widetilde \eta_\theta^\circ$ is the tensor product $\theta^\circ \otimes H_c^{(n-1)(h-k)^+}(X_h, \overline \QQ_\ell)[\chi]$ where $\theta^\circ(z,g) = \theta(z)$ for $(z,g) \in \langle \zeta \rangle \ltimes \UnipF = \cR_{h,n,q}^\times(\F)$. Finally, the trace identity is a special case of Proposition \ref{p:vregtrace}.

Let $\widetilde X_h \colonequals \sqcup_m \widetilde X_h^{(m)}$. The action of $L^\times \times D^\times$ on $\widetilde X$ induces an action of $G \colonequals (L^\times/U_L^h) \times (D^\times/U_D^{(h)})$ on $\widetilde X_h$. Moreover, $H_*(\widetilde X, \overline \QQ_\ell)[\theta] \subset H_*(\widetilde X_h, \overline \QQ_\ell)$, so it is enough to understand the cohomology of $\widetilde X_h$. By construction, it is easy to see that $\widetilde X_h$ is equal to the $G$-translates of $\iota_h(X_h) \subset \widetilde X_h^{(0)}$. One can define an action of
\begin{equation*}
\Gamma = \langle (\pi, \pi^{-1}) \rangle \cdot \langle (\zeta, \zeta^{-1}) \rangle \cdot (U_L^1/U_L^h \times U_D^1/U_D^{(h)}) \subset G
\end{equation*}
on $X_h$ so that $\iota_h$ is $\Gamma$-equivariant. Moreover, the stabilizer of $\iota_h(X_h)$ in $G$ is exactly equal to $\Gamma$. The claim in (b) then follows from an analysis of the $\theta$-eigenspace of $\Ind_\Gamma^G(H_i(X_h, \overline \QQ_\ell)).$

For any $x \in U_L^{h-1}$, we have $\eta_\theta'(x) = \psi(x)$ and
\begin{equation*}
\eta_\theta'(\Pi \cdot x \cdot \Pi^{-1}) = \eta_\theta'(\varphi(x)) = \psi(x^q).
\end{equation*}
Since $\theta$ is primitive, it follows that the normalizer of $\eta_\theta'$ in $D^\times$ is equal to $\pi^\ZZ \cdot \cO_D^\times$. Irreducibility then follows by Mackey's criterion. The dimension of the $\eta_\theta$ is equal to the product of the index $[D^\times : \pi^\ZZ \cdot \cO_D^\times] = n$ and the dimension of $\eta_\theta'$, so the desired result holds by Theorem \ref{t:irrepdesc}.
\end{proof}

\begin{remark}
Note that if $h \leq k$, then $\cO_D^\times/U_D^{(h)} = \cO_L^\times/U_L^h$, so the character $\theta \from L^\times \to \QQ_\ell^\times$ can be viewed as a one-dimensional representation of $\pi^\ZZ \cdot \cO_D^\times$. By Proposition \ref{p:divalg},
\begin{flalign*}
&&H_0(\widetilde X, \overline \QQ_\ell)[\theta] \cong \Ind_{\pi^\ZZ \cdot \cO_D^\times}^{D^\times}(\theta). &&\Diamond
\end{flalign*} 
\end{remark}

\subsection{Local Langlands correspondences} \label{s:LLC}

Fix a character $\epsilon$ of $K^\times$ whose kernel is equal to the image of the norm $\Norm_{L/K} \from L^\times \to K^\times$. Then the representation
\begin{equation*}
\sigma_\theta \colonequals \Ind_{\cW_L}^{\cW_K}(\theta \circ \rec_L)
\end{equation*}
is a smooth irreducible $n$-dimensional representation of $\cW_K$. Let $\sX$ denote the set of all characters of $L^\times$ that have trivial stabilizer in $\Gal(L/K)$ and let $\cG_K^\epsilon(n)$ denote the set of (isomorphism classes of) smooth irreducible $n$-dimensional representations $\sigma$ of $\Weil_K$ that satisfy $\sigma \cong \sigma \otimes (\epsilon \circ \rec_F).$ Let $\xi$ be the character of $L^\times$ determined by $\xi(\pi) = (-1)^{n-1}$ and $\xi|_{\cO_L^\times} = 1$. Then
\begin{center}
\begin{tikzpicture}[xscale=4,yscale=0.65]
\draw (0,0) node(a){$\sX/\Gal(L/K)$} (1,0) node(b){$\cG_K^\epsilon(n)$};
\draw[->] (a.east) to node[above]{LCFT} (b.west);
\draw (0,-1) node(a'){$\theta$} (1,-1) node(b'){$\sigma_{\xi\theta}$};
\draw[|->] (a'.east) to (b'.west);
\end{tikzpicture}
\end{center}
is a bijection. This is a twisted version of Lemma 1.1 of \cite{BW13}. (Here, LCFT stands for local class field theory.)

Now let $\cA_K^\epsilon(n)$ denote the set of (isomorphism classes of) irreducible supercuspidal representations $\pi$ of $\GL_n(K)$ such that $\pi \cong \pi \otimes (\epsilon \circ \det)$. There exists a canonical bijection
\begin{center}
\begin{tikzpicture}[xscale=4,yscale=0.65]
\draw (0,0) node(a){$\cG_K^\epsilon(n)$} (1,0) node(b){$\cA_K^\epsilon(n)$};
\draw[->] (a.east) to node[above]{\text{LLC}} (b.west);
\draw (0,-1) node(a'){$\sigma_{\xi\theta}$} (1,-1) node(b'){$\pi_\theta$};
\draw[|->] (a'.east) to (b'.west);
\end{tikzpicture}
\end{center}
known as the local Langlands correspondence.

Finally, let $\cA'{}_{\!\!K}^\epsilon(n)$ denote the set of (isomorphism classes of) irreducible representations $\rho$ of $D^\times$ such that $\rho \cong \rho \otimes (\epsilon \circ \Nrd_{D/K})$. Then the Jacquet--Langlands correspondence gives a bijection
\begin{center}
\begin{tikzpicture}[xscale=4,yscale=0.65]
\draw (0,0) node(a){$\cA_K^\epsilon(n)$} (1,0) node(b){$\cA'{}_{\!\!K}^\epsilon(n)$};
\draw[->] (a.east) to node[above]{\text{JLC}} (b.west);
\draw (0,-1) node(a'){$\pi_\theta$} (1,-1) node(b'){$\rho_\theta$};
\draw[|->] (a'.east) to (b'.west);
\end{tikzpicture}
\end{center}

\begin{remark}\label{r:epsiloninvar}
Since $L/K$ is unramified, the restriction of $\epsilon$ to $\OH_K^\times$ is trivial, and thus the composition $\epsilon \circ \Nrd_{D/K}$ is trivial on $E^\times \cdot \OH_D^\times \supset \pi^\ZZ \cdot \OH_D^\times$. Thus by the construction of $\eta_\theta$, we have that $\eta_\theta$ is invariant under twisting by $\epsilon \circ \Nrd_{D/K}$. \hfill $\Diamond$
\end{remark}

Our work describes a correspondence between $L^\times$-characters and $D^\times$-representations given by
\begin{center}
\begin{tikzpicture}[xscale=8]
\draw (0,0) node(a){\{primitive characters of $L^\times$\}} (1,0) node(b){\{irreducible representations of $D^\times$\}};
\draw[->] (a.east) to node[above]{\text{$p$-adic DL}} (b.west);
\draw (0,-1) node(a'){$\theta$} (1,-1) node(b'){$\eta_\theta \colonequals H_\bullet(\widetilde X, \overline \QQ_\ell)[\theta]$};
\draw[|->] (a'.east) to (b'.west);
\end{tikzpicture}
\end{center}
By Remark \ref{r:epsiloninvar}, we see that $\eta_\theta \in \cA'{}_{\!\!K}^\epsilon(n)$. In Theorem \ref{t:divalg}, we prove that this correspondence matches the composition of the previous three, therefore giving a geometric realization of the Jacquet--Langlands correspondence.

\begin{remark}
The construction of the local Langlands and Jacquet--Langlands correspondences was already known. See, for example, \cite{H93}. Recent work of Boyarchenko and Weinstein (see \cite{BW13}) gives a partially geometric construction of these correspondences using the representations $H_c^{n-1}(X_2, \overline \QQ_\ell)[\psi]$ of $U_2^{n,q}(\FF_{q^n})$. Note that in \cite{BW14} and \cite{BW13}, the scheme $X_2$ is denoted by $X$ and the group $U_2^{n,q}(\F)$ is denoted by $U^{n,q}(\F)$. The following theorem shows that Deligne--Lusztig constructions for division algebras give a geometric realization of the Jacquet--Langlands correspondence. \hfill $\Diamond$
\end{remark}

\begin{theorem}\label{t:divalg}
Let $\theta \from L^\times \to \overline \QQ_\ell^\times$ be a primitive character of level $h$ and let $\rho_\theta$ be the $D^\times$-representation corresponding to $\theta$ under the local Langlands and Jacquet--Langlands correspondences. Then $H_i(\widetilde X, \overline \QQ_\ell)[\theta] = 0$ if $i \neq (n-1)(h-k)^+$ and 
\begin{equation*}
H_{(n-1)(h-k)^+}(\widetilde X, \overline \QQ_\ell)[\theta] \cong \rho_\theta.
\end{equation*}
\end{theorem}

\begin{proof}
By Proposition 1.5(b) of \cite{BW13}, we just need to show that $\eta_\theta \colonequals H_{(n-1)(h-k)^+}(\widetilde X, \overline \QQ_\ell)[\theta]$ satisfies the following two properties:
\begin{enumerate}[label=(\roman*)]
\item
For any character $\epsilon$ of $K^\times$ whose kernel is equal to the image of the norm map $\Norm_{L/K} \from L^\times \to K^\times$, we have $\eta_\theta \cong \eta_\theta \otimes (\epsilon \circ \Nrd_{D/K})$.

\item
There exists a constant $c$ such that $\tr \eta_\theta(x) = c \cdot \sum_{\gamma \in \Gal(L/K)} \theta^\gamma(x)$ for each very regular element $x \in \OH_L^\times$.
\end{enumerate}

Since $L/K$ is unramified, the restriction of $\epsilon$ to $\cO_K^\times$ is trivial, and thus the composition $\epsilon \circ \Nrd_{D/K}$ is trivial on $L^\times \cdot \cO_D^\times \supset \pi^\ZZ \cdot \cO_D^\times$. Thus by construction, $\eta_\theta$ is invariant under twisting by $\eta \circ \Nrd_{D/K}$. This proves (i).

We now prove (ii). By the construction of $\eta_\theta$, since $\pi^\ZZ \cdot \OH_D = L^\times \cdot U_D^1$, we have
\begin{equation*}
\tr \eta_\theta(x) = \sum_{\substack{g \in D^\times/L^\times \cdot U_D^1 \\ gxg^{-1} \in L^\times \cdot U_D^1}} \tr \eta_\theta'(gxg^{-1}).
\end{equation*}
Now let $x \in \OH_L^\times$ be very regular. By Proposition \ref{p:vregtrace}, $\eta_\theta^\circ(x) = (-1)^{(n-1)(h-k)^+} \theta(x)$. By Lemma 5.1(b) of \cite{BW13}, if $g \in D^\times$ is such that $gxg^{-1} \in L^\times \cdot U_D^1$, then $g \in N_{D^\times}(L^\times) \cdot U_D^1$, where $N_D^\times(L^\times)$ is the normalizer of $L^\times$ in $D^\times$. Therefore
\begin{align*}
\tr \eta_\theta(x) 
&= \sum_{g \in N_{D^\times}(L^\times) \cdot U_D^1/L^\times \cdot U_D^1} \tr \eta_\theta'(gxg^{-1}) = \sum_g \tr(\eta_\theta^\circ(gxg^{-1})) \\
&= \sum_g (-1)^{(n-1)(h-k)^+} \theta(gxg^{-1}) = (-1)^{(n-1)(h-k)^+} \cdot \sum_{\gamma \in \Gal(L/K)} \theta^\gamma(x). \qedhere
\end{align*}
\end{proof}

\begin{corollary}\label{c:JL}
Let $D$ and $D'$ be division algebras of rank $n$ and let $X$ and $X'$ be their corresponding Deligne--Lusztig constructions. For any primitive character $\theta \from L \to \overline \QQ_\ell^\times$, the Jacquet--Langlands transfer of $H_{(n-1)(h-k)^+}(X)[\theta]$ is isomorphic to $H_{(n-1)(h-k')}(X)[\theta]$, where $k/n$ and $k'/n$ are the Hasse invariants of $D$ and $D'$.
\end{corollary}


\begin{thebibliography}{WW99}
\bibitem[B12]{B12}
Boyarchenko, M. \textit{Deligne--Lusztig constructions for unipotent and $p$-adic groups.} Preprint, 2012. arXiv:1207.5876.

\bibitem[BW13]{BW13} Boyarchenko, M. and Weinstein, J. \textit{Geometric realization of special cases of local Langlands and Jacquet--Langlands correspondences.} Preprint, arXiv:1303.5795, 2013.

\bibitem[BW14]{BW14} Boyarchenko, M. and Weinstein, J. \textit{Maximal varieties and the local Langlands correspondence for $\GL_n$.} To appear in \textit{Journal of the AMS,} 2014. arXiv:1109.3522v3.

\bibitem[C14]{C14} Chan, C. \textit{Deligne--Lusztig constructions for division algebras and the local Langlands correspondence}. Preprint, arXiv:1406.6122, 2014.

\bibitem[C74]{C74} Corwin, L.  \textit{Representations of division algebras over local fields.} {Advances in Mathematics,} 13, 259-267, 1974.



\bibitem[DL76]{DL76} Deligne, P. and Lusztig, G. \textit{Representations of reductive groups over finite fields.} {Ann.\ of Math.\ (2),} 103(1):103-161, 1976.

\bibitem[FF13]{FF13} Fargues, L. and Fontaine, J-M. \textit{Courbes et fibr\'es vectoriels en th\'eorie de Hodge $p$-adique}. Preprint, 2013. {http://webusers.imj-prg.fr/$\sim$laurent.fargues/Prepublications.html}.

\bibitem[H93]{H93} Henniart, G. \textit{Correspondance de Jacuqet--Langlands explicite. I. Le case mod\'er\'e de degr\'e premier.} In {S\'eminaire de Th\'eorie des Nombres,} Paris, 1990-91, volume 108 of Prog.\ Math., pages 85-114. 1993.


\bibitem[I15]{I15} Ivanov, A. \textit{Affine Deligne--Lusztig varieties of higher level and the local Langlands correspondence for $\GL_2$}. Preprint, 2015. arXiv:1504.00264.

\bibitem[L79]{L79} Lusztig, G. \textit{Some remarks on the supercuspidal representations of $p$-adic semisimple groups.} In {Automorphic forms, representations and L-functions (Proc.\ Sympos.\ Pure Math., Oregon State Univ., Corvallis, Ore., 1977), Part 1}, Proc.\ Sympos.\ Pure Math., XXXIII, pages 171-175. Amer.\ Math.\ Soc., Providence, R.I., 1979.

\bibitem[L04]{L04} Lusztig, G. \textit{Representations of reductive groups over finite rings.}


\end{thebibliography}
\end{document}